\documentclass[12pt,reqno]{amsart}
\usepackage{latexsym,amsmath,amsfonts,amscd,amssymb,color}
\usepackage{graphicx,wrapfig}
\usepackage[all]{xy}

\newcommand{\inc}{\hookrightarrow}
\newcommand{\la}{\langle}
\newcommand{\ra}{\rangle}
\newcommand{\x}{\times}

\newcommand{\bd}{\partial}

\newcommand{\ZZ}{\mathbb{Z}}
\newcommand{\CC}{\mathbb{C}}

\newcommand{\CP}{\mathbb{C}P}
\newcommand{\RR}{\mathbb{R}}

\newcommand{\im}{\operatorname{im}}

\newcommand{\U}{\operatorname{U}}

\newcommand{\OO}{\operatorname{O}}   
\newcommand{\Sp}{\operatorname{Sp}}  
\newcommand{\GL}{\operatorname{GL}}

\DeclareMathOperator{\coker}{coker}
\DeclareMathOperator{\Id}{Id}
\DeclareMathOperator{\Diff}{Diff}

\DeclareMathOperator{\End}{End}
 
\DeclareMathOperator{\vol}{vol}
\DeclareMathOperator{\fix}{fix}

\renewcommand{\a}{\alpha}
\renewcommand{\b}{\beta}
\renewcommand{\d}{\delta}
\newcommand{\g}{\gamma}
\newcommand{\e}{\varepsilon}

\newcommand{\f}{\varphi}
\newcommand{\h}{\theta}
\renewcommand{\l}{\lambda}
\renewcommand{\k}{\kappa}

\newcommand{\n}{\nu}
\renewcommand{\o}{\omega}
\newcommand{\p}{\phi}
\newcommand{\q}{\psi}
\renewcommand{\t}{\tau}

\newcommand{\G}{\Gamma}
\renewcommand{\O}{\Omega}
\renewcommand{\S}{\Sigma}
\newcommand{\D}{\Delta}
\renewcommand{\L}{\Lambda}

\newcommand{\ii}{\mathrm{i}}
\newcommand{\rId}{\mathrm{Id}} 

\newcommand{\frj}{\mathfrak{j}} 

\newcommand{\cA}{\mathcal{A}} 
\newcommand{\cV}{\mathcal{V}} 
\newcommand{\cH}{\mathcal{H}} 

\newtheorem{theorem}{Theorem}
\newtheorem{proposition}[theorem]{Proposition}
\newtheorem{lemma}[theorem]{Lemma}
\newtheorem{definition}[theorem]{Definition}

\newtheorem{corollary}[theorem]{Corollary}
\newtheorem{remark}[theorem]{Remark}
\newtheorem{remarks}[theorem]{Remarks}

\title{Resolution of $4$-dimensional symplectic orbifolds}

\author[L. Mart\' in-Merch\' an]{Luc\'ia  Mart\' in-Merch\' an}
\address{Facultad de Ciencias Matem\'aticas, Universidad
Complutense de Madrid, Plaza de Ciencias 3, 28040 Madrid, Spain}
\email{lmmerchan@ucm.es}

\author[J. Rojo]{Juan Rojo}
\address{ETS Ingenieros Inform\'aticos, Universidad Polit\'ecnica de Madrid,
Campus de Montegancedo, 28660, Madrid, Spain}
\email{juan.rojo.carulli@upm.es}



\begin{document}

\begin{abstract}
We give a method to resolve $4$-dimensional symplectic orbifolds
making use of techniques from complex geometry and gluing of symplectic forms.
We provide some examples to which the resolution method applies.
\end{abstract}

\maketitle


\section{Introduction}

An orbifold is a space which is locally modelled on balls of $\RR^n$ quotient by a finite group.
These have been very useful in many geometrical contexts \cite{Th}. In the setting of 
symplectic geometry, symplectic orbifolds have been introduced mainly as a way to
construct symplectic manifolds by resolving their singularities.
The problem of resolution of singularities and blow-up in the symplectic setting 
was posed by Gromov in \cite{G}. Few years later, the symplectic blow-up was
rigorously defined by McDuff \cite{McD} and it was used to construct a simply-connected
symplectic manifold with no K\"ahler structure.
The concept of symplectic blow-up was later on generalized to the orbifold setting in \cite{God}.

McCarthy and Wolfson developed in \cite{MW} a symplectic resolution for isolated 
singularities of orbifolds in dimension $4$.
Later on, Cavalcanti, Fern\'andez and Mu\~noz gave a method of performing 
symplectic resolution of \emph{isolated} orbifold singularities in all dimensions \cite{CFM}. This was
used in \cite{FM} to give the first example of a simply-connected symplectic $8$-manifold
which is non-formal, as the resolution of a suitable symplectic $8$-orbifold. This manifold
was proved to have also a complex structure in \cite{BM}.

Bazzoni, Fern\'andez and Mu\~noz \cite{BFM} have given the first construction of a 
symplectic resolution of an orbifold of dimension $6$ with isotropy sets of dimension $0$ and
$2$, although the construction is ad hoc for the particular example at hand as it satisfies that
the normal bundle to the $2$-dimensional isotropy set is trivial. This was used to give
the first example of a simply-connected non-K\"ahler manifold which is simultaneously complex
and symplectic.

Niederkr\"uger and Pasquotto provided methods 
for resolving  different types of symplectic
orbifold singularities in \cite{NP1,NP2}. The first deals with orbifolds
arising as symplectic reductions of Hamiltonian circle actions; these singularities are cyclic and
might not be isolated.
In dimension $4$, the previous work in \cite{MR} serves to resolve symplectic $4$-orbifolds
whose isotropy set consists of codimension $2$ \emph{disjoint} submanifolds. 
In such case the orbifold is topologically a manifold
(the isotropy points are non-singular),
so the question only amounts to change the orbifold symplectic form into a smooth symplectic form.

In this paper we give an elementary and self-contained method to resolve arbitrary symplectic $4$-orbifolds.
For the symplectic part, we make use of techniques for gluing symplectic forms.
These include the so called \emph{inflation procedure} introduced by Thurston in \cite{Th},
and the notion of \emph{positivity} (or tameness) with respect to an almost complex structure,
studied in detail in the book \cite{MS}.
For the topological part (the resolution of quotient singularities), 
we mainly make use of complex local models from \cite{CFM}, and tools coming from Invariant Theory.
There is however an essential difficulty when dealing with 
non-isolated isotropy points; this comes from the fact
that the (local) resolution of the topologically-singular points must be made compatible with the resolution
of the isotropy divisors (real codimension $2$) of the orbifold.
To overcome this difficulty, the desingularization of the isotropy divisors has to be made with care. 
The method in \cite{MRT} starts with a manifold and constructs on it an orbifold atlas
with isotropy along a configuration of divisors. 
This construction has to be reversed, but with an essential change:
mainly, that the orbifold and the manifold structures along the divisors must be related through
an \emph{holomorphic} map.

The main result is:

\begin{theorem} \label{thm:main-thm}
Let $(X, \o)$ be a compact symplectic $4$-orbifold.
There exists a symplectic manifold $(\tilde{X}, \tilde{\o})$ and a smooth map 
$\pi: (\tilde{X},\tilde{\o}) \to (X,\o)$
which is a symplectomorphism outside an arbitrarily small neighborhood of the isotropy set of $X$.
\end{theorem}
Actually, the compactness hypothesis in the above Theorem can be relaxed: it suffices
that every connected component $S \subset X$ of the set of isotropy surfaces has compact closure $\bar{S}$ in $X$.

In addition, Theorem \ref{thm:main-thm} can be used to construct a $4$-dimensional simply connected symplectic manifold
as the symplectic resolution of a suitable $4$-orbifold. 
This symplectic orbifold is a quotient of a K\" ahler manifold $M_\g(\S_2) \times S^1$ 
by an action of $\ZZ_2 \x \ZZ_2$, where $M_\g(\S_2)$ is a non-trivial mapping torus of the genus $2$ surface. 
The isotropy set of the action consists of $8$ isolated points and $3$ tori that have $4$ intersection points, 
so this symplectic orbifold cannot be resolved with the methods of \cite{CFM,MR}.

In the recent paper \cite{C}, it is given an alternative method for resolving arbitrary symplectic $4$-orbifolds.
The techniques used in \cite{C} (e.g. symplectic reduction and symplectic fillings)
differ completely from ours and are technically more involved.

This paper is organized as follows. In section \ref{sec:sym-orb} we review the necessary preliminaries on symplectic orbifolds. Section \ref{sec:sym-orb-4} studies the isotropy set of $4$-dimensional symplectic orbifolds, giving special local models for the isotropy surfaces. 
With these tools at hand we prove Theorem \ref{thm:main-thm} in section \ref{sec:resolution}.
Finally, in section \ref{sec:examples} we provide some examples in which the symplectic resolution
of Theorem \ref{thm:main-thm} applies.

\noindent\textbf{Acknowledgements.} We are grateful to Vicente Mu\~noz and Giovanni Bazzoni for useful conversations.
The first author acknowledges financial support by a FPU Grant(FPU16/03475). 

\section{Symplectic orbifolds} \label{sec:sym-orb}
In this section we introduce some aspects about orbifolds and symplectic orbifolds, which can be found in \cite{MR}.

\subsection{Orbifolds}

\begin{definition}\label{def:orbifold}
An $n$-dimensional orbifold is a Hausdorff and second countable space $X$ 
endowed with an atlas $\{(U_{\a},V_\a, \phi_{\a},\G_{\a})\}$, 
where $\{V_\a\}$ is an open cover of $X$,
$U_\a \subset\RR^n$, $\G_{\a} < \Diff(U_\a)$ is a finite group acting by diffeomorphisms, and 
$\phi_{\a}:U_{\a} \to V_{\a} \subset X$ is a $\G_{\a}$-invariant map which induces a 
homeomorphism $U_{\a}/\G_\a \cong V_{\a}$. 

There is a condition of compatibility of charts for intersections.
For each point $x \in V_{\a} \cap V_{\b}$ there is some $V_\d \subset V_{\a} \cap V_{\b}$ 
with $x \in V_\d$ so that there are group monomorphisms $\rho_{\d \a}: \G_\d \inc \G_\a$,
$\rho_{\d \b}: \G_\d \inc \G_\b$, and open differentiable embeddings $\imath_{\d \a}: U_\d \to U_\a$, 
$\imath_{\d \b}: U_\d \to U_\b$, which satisfy 
$\imath_{\d \a}(\g(x))=\rho_{\d \a}(\g)(\imath_{\d \a}(x))$ and 
$\imath_{\d \b}(\g(x)) = \rho_{\d \b}(\g)(\imath_{\d \b}(x))$, for all $\g \in \G_\d$. 
\end{definition}

The concept of change of charts in orbifolds is borrowed from its analogue in manifolds.
\begin{definition} \label{def:change of charts}
For an orbifold $X$, a \emph{change of charts} is the map 
 $$
\psi^{\d}_{\a \b}= \imath_{\d \b} \circ \imath_{\d \a}^{-1}:\imath_{\d \a}(U_\d)
\to \imath_{\d \b}(U_\d).
 $$
\end{definition}
Note that $\imath_{\d \a}(U_\d) \subset U_\a$ and $\imath_{\d \b}(U_\d) \subset U_\b$,
so $\psi^{\d}_{\a \b}$ is a change of charts from $U_\a$ to $U_\b$.
Clearly a change of charts between $U_\a$ and $U_\b$ depends on the 
inclusion of a third chart $U_\d$. This dependence is up to the action of an element in $\G_\d$. 
In general this dependence is irrelevant, so
we may abuse notation and write $\psi_{\a \b}$ for any change of chart between $U_\a$ and $U_\b$.
We may further abuse notation and write
$$
\psi_{\a \b}: U_\a \to U_\b
$$
for a change of charts as above, even though its domain and range do not equal in general
all $U_\a$ and $U_\b$ but an open subset of them.

\medskip
 
We can refine the atlas of an orbifold $X$ in order to obtain better properties; given a point $x \in X$, 
there is a chart $(U,V,\phi,\G)$ with $U \subset \RR^n$, $U/\G \cong V$,
so that the preimage $\phi^{-1}(\{x\})= \{u\}$ is only a point,
and the group $\G$ acting on $U$ leaves the point $u$ fixed, 
i.e. $\g(u)=u$ for all $\g \in \G$. We call $\G$ the \emph{isotropy group} at $x$, 
and we denote it by $\G_x$. This group is well defined up to conjugation by a diffeomorphism
of a small open set of $\RR^n$.
In addition, using a $\G_x$-invariant metric and the exponential chart one can prove:

\begin{proposition}
Around any point $x \in X$ 
there exists an orbifold chart $(U,V,\p,\G)$ with $\G_x=\G < \OO(n)$.
\end{proposition}

\begin{definition}
The \emph{isotropy subset} of $X$ is 
$
\S=\{x\in X \mbox{ s.t. } \G_x\neq \emptyset\}.
$
\end{definition}

As we shall see, the isotropy set is stratified into suborbifolds; 
this notion is also similar to the concept of a submanifold:

\begin{definition} \label{def:suborbifold}
Let $X$ be an orbifold of dimension $n$.
A \emph{suborbifold of dimension $d$} or \emph{d-suborbifold} of $X$ is defined to be a subspace $Y \subset X$ such that
for each $p \in Y$ there exists an orbifold chart $(U,V,\phi,\G)$ of $X$ around $p$
with $\G < \OO(n)$, $\phi(p)=0$, and such that $U' = U \cap (\RR^d \x \{0\})$ satisfies $\phi(U')=Y \cap V$.
\end{definition}

Let $Y \subset X$ be a suborbifold. Then $Y$ has a structure of orbifold inherited from $X$, as follows.
Consider the chart $(U,V,\phi,\G)$ of the above definition and let us identify
$\RR^d \cong \RR^d \x \{0\} \subset \RR^n$. Consider $\tilde{\G}=\{ \g \in \G \mbox{ s.t.} \g(\RR^d) \subset \RR^d\} < \G$
the subgroup of elements leaving invariant $\RR^d$. Consider the representation given by
$\varrho:\widetilde{\G} \to \End(\RR^d)$; its image is a subgroup 
$\G' = \im(\varrho) \cong \widetilde{\G}/\ker (\varrho)$. 
Let us denote $V'=Y \cap V=\phi(U')$, and $\p'=\p|_{U'}:U' \to V'$.
The orbifold chart of $Y$ around $p$ is defined to be $(U',V',\p',\G')$.
Clearly, $U'$ is a $\G'$-invariant set and satisfies $U'/\G' \cong Y \cap V$.

Let us state a notion of equivalence between groups of diffeomorphisms that is useful for orbifolds.
\begin{definition}
Let $H < \Diff(U)$, $H' < \Diff(U')$ 
be two groups of diffeomorphisms of open sets $U, U'$ of $\RR^{2n}$.
We say that the germs $(U,H)$ and $(U',H')$ are \emph{equivalent} if there exists
a diffeomorphism $f \colon U \to U'$ such that $f \circ H \circ f^{-1} = H'$.
In this case we write $(U,H) \cong (U',H')$.
\end{definition}
Note that the above gives an equivalence relation in the set of germs 
of diffeomorphisms of $\RR^{2n}$. 
If $(U,V,\G,\p)$ is an orbifold chart, 
a diffeomorphism $f \colon U \to U'$ gives an induced orbifold chart
$(U',V, \G', \p')$, where $\G'=f \circ \G \circ f^{-1}$ and $\p'=\p \circ f^{-1}$.
Hence, all the germs $(U',\G')$ equivalent to $(U,\G)$ induce the same
orbifold chart.  We shall also specify this notion for soubgroups of $\OO(n)$.

\begin{definition}
Two subgroups $\G, \G'$ of $\OO(n)$
are \emph{equivalent} if there exists open sets $U, U' \subset \RR^n$ containing $0$
such that the germs $(U,\G)$ and $(U',\G')$ are equivalent.
We denote $\G \cong \G'$ in this case.
\end{definition}

\begin{proposition} \cite[Proposition 4]{MR} \label{prop:isotropy set}
Let $X$ be an orbifold, and let $\S$ be its isotropy subset. For every equivalence class $H$ of finite subgroup $H<\OO(n)$, we can define the set
$$ 
\S_H= \{ x\in X \mbox{ s.t. } \G_x \cong H\}.
$$
Then the closure $\overline{\S}_H$ is a suborbifold of $X$, and 
$\S_H= \overline{\S}_H - \bigcup_{H<H'} \S_{H'}$ is a submanifold of $X$.
\end{proposition}

\begin{definition}
An orbifold function $f \colon X\to \RR$ is a continuous function such that $f\circ \phi_\a \colon U_\a \to \RR$ is smooth
for every $\a$. 
\end{definition}

Note that this is equivalent to giving smooth functions $f_\a$ on $U_\a$ which are $\G_\a$-equivariant
and which agree under the changes of charts.
An \emph{orbifold partition of unity subordinated to the open cover $\{V_\a\}$ of $X$} consists of orbifold functions
$\rho_\a \colon X \to [0,1]$ such that
the support of $\rho_\a$ lies inside $V_\a$ and the sum $\sum_\a \rho_\a \equiv 1$ on $X$.

\begin{proposition} \label{prop:partition unity}  \cite[Proposition 5]{MR}
Let $X$ be an $n$-orbifold. For any sufficiently refined open cover $\{V_\a\}$ of $X$ there exists an
orbifold partition of unity subordinated to $\{V_\a\}$.
\end{proposition}

Orbifold tensors are defined in the same way as functions are. 
That is, an orbifold tensor on $X$ is a collection of 
$\G_\a$-invariant tensors on each $U_\a$ which agree under the changes of charts. 
In particular, there is a notion of orbifold differential forms $\Omega_{orb}(X)$ and
the exterior derivative is also well-defined.

\subsection{Symplectic orbifolds}

\begin{definition}
A symplectic orbifold is an orbifold $X$ equipped with and orbifold $2$-form $\o \in \O^2_{orb}(X)$ such that $d\o=0$ and $\o^n>0$ where $2n=\dim(X)$.
\end{definition}

The proof of the existence of an almost K\" ahler structure on a manifold (see \cite{Silva}) easily carries over to the orbifold case: 

\begin{proposition} \cite[Proposition 8]{MR}\label{prop:almost-Kahler}
Let $(X,\o)$ be a symplectic orbifold. Then $(X,\o)$ admits 
an almost K\"ahler orbifold structure $(X, \o, J , g)$.
\end{proposition}

\begin{corollary} \label{cor:isotropy unitary}
Let $(X,\o)$ be a symplectic $2n$-orbifold. Every point in $X$ admits a chart
$(U,V,\phi,\G,\o)$ with $\G < \U(n)$.
\end{corollary}
\begin{proof}
Put any almost K\"ahler structure $(\o,J,g)$ on $X$ as provided by Proposition \ref{prop:almost-Kahler}.
Fix a chart $(U,V,\phi,\G)$ around $p$ such that $\phi(0)=p$, $\G$ acts linearly, 
and the almost K{\"a}hler structure $(\o_p,J_p,g_p)=(\o_0,\frj,g_0)$ at $p=0$ is standard.
As $J$ is an orbifold almost complex structure, $\G$ preserves $J$; in particular
at the point $0 \in U$ we have $d_0\g \circ \frj=\frj \circ d_0\g$ for all $\g\in \G$. 
As $\g$ is linear, we have that $d_0\g=\g$, hence $\g$ preserves the complex structure of $\CC^n=(\RR^{2n},\frj)$. This means that $\G<\GL(n,\CC)$.
Analogously, since $\g$ preserves the standard metric $g_0$, one sees that $\G <\OO(2n)$.
The conclusion is that $\G < \GL(n,\CC) \cap \OO(2n)=\U(n)$.
\end{proof}

For symplectic (almost K{\"a}hler) orbifolds, the isotropy set inherits a symplectic (almost K{\"a}hler) structure.

\begin{corollary} \cite[Corollary 9]{MR}\label{cor:overSH}
The isotropy set $\S$ of $(X,\o)$ consists of immersed symplectic 
suborbifolds $\overline{\S}_H$.
Moreover, if we endow $X$ with an almost K\"ahler orbifold structure $(\o,J,g)$, 
then the sets $\overline{\S}_H$ are almost K\"ahler suborbifolds.
\end{corollary}

The following result is a Darboux theorem for symplectic orbifolds.

\begin{proposition} \cite[Proposition 10]{MR} \label{prop:Darboux}
Let $(X,\o)$ be a symplectic orbifold and $x_0 \in X$. There exists an orbifold chart $(U, V,\phi, \G)$ around $x_0$
with local coordinates $(x_1,y_1,\ldots, x_n,y_n)$ such that the symplectic form 
has the expression $\o= \sum dx_i \wedge dy_i$ and $\G < \U(n)$ is a subgroup of the unitary group.
\end{proposition}

Any orbifold almost K{\"a}hler structure 
can be perturbed to make it standard around any chosen point.
We include a proof below.
Denote $\frj$ the standard complex structure on $\CC^n$.

\begin{corollary} \label{cor:flat chart}
Let $(X,\o)$ be a symplectic orbifold, and let $(J,g)$ be a compatible 
almost K{\"a}hler structure. Let $p \in X$ a point and $(U,V,\phi,\G,\o_0)$ 
a Darboux chart around $p$.
Choose $V_1$ a neighborhood of $p$ such that $\overline{V_1} \subset V$,
and let $U_1=\phi^{-1}(V_1) \subset U$.
There exists another compatible almost K{\"a}hler structure $(J',g')$
such that $J'=J$ and $g'=g$ outside $V$, and $(J',g')$ is the standard $(\frj,g_0)$
when lifted to the chart $U_1 \subset U$.
\end{corollary}
\begin{proof}
Take a bump function $\rho$ which equals $1$ in $V_1$ and $0$ outside $V$.
Consider the metric $g_1=\rho g_0 + (1-\rho)g$, where $\rho g_0$ coincides
with the standard metric in $U_1$ and extends as $0$ to all $X$.
If we use the metric $g_1$ as auxiliary metric in the proof of Proposition 
\ref{prop:almost-Kahler} and construct a compatible almost K{\"a}hler structure 
$(g',J')$, we find that $J'=\frj$, $g'=g_0$ when lifted in $U_1$ because both $\o$
and the auxiliary metric $g_1$ are standard in $U_1$.
\end{proof}

Let us recall a result from symplectic linear algebra that will be useful later.
Consider the retraction
\begin{equation} \label{eqn:r(A)}
 r \colon \Sp(2n,\RR)\to \U(n), \qquad r(A)=A(A^tA)^{-1/2}
\end{equation}
The fact that $A(A^tA)^{-1/2} \in \U(n)=\Sp(2n,\RR) \cap \OO(2n)$ for any $A \in \Sp(2n,\RR)$
can be seen as follows. First, since $A^t \O_0 A=\O_0$, with $\O_0$ the matrix of the standard symplectic form on $\RR^{2n}$, it is easy to check that $(A^tA)^t \O_0 A^tA=\O_0$
so $A^tA \in \Sp(2n,\RR)$. Then, by expressing the square root $S^{1/2}$ as a power series of $S$, for $S$
a positive definite symmetric matrix, one sees that $(A^t A)^{1/2} \in \Sp(2n,\RR)$, hence so does its inverse
and it follows that $r(A)=A(A^tA)^{-1/2} \in \Sp(2n,\RR)$. Finally, using that $S$ and $S^{1/2}$ commute,
it follows that $r(A)^t r(A)=\Id$, so $r(A) \in \OO(2n)$.
 
This retraction satisfies the following. 
If there is a group $\G<\U(k)$ and an isomorphism $\rho \colon \G \to \G'<\U(k)$, such that
$A \in \Sp(2n,\RR)$ is $\G$-equivariant in the sense that $A \circ \g=\rho(\g)\circ A$
for all $\g \in G$, then $r(A)$ is also $\G$-equivariant, i.e. 
$r(A) \circ \g=\rho(\g)\circ r(A)$ for all $\g \in G$. 
This property is a consequence of the following result:

\begin{lemma} \cite[Lemma 21]{MR} \label{lem:retrac}
Let $A,C \in \U(k)$ and $B \in \Sp(2n,\RR)$ such that $A=B^{-1}  C B$.
Then $A=r(B)^{-1} C \, r(B)$.
\end{lemma}

\section{Symplectic orbifolds in dimension $4$.} \label{sec:sym-orb-4}

\subsection{The isotropy set in dimension $4$.} \label{subsec:local}

Let $(X,\o)$ be a symplectic orbifold of dimension $4$. Let $x \in X$ a point.
Put a compatible orbifold almost complex structure on $(X,\o)$,
obtaining an almost K{\"a}hler orbifold $(X,\o,J)$.
By the equivariant Darboux Theorem, around any point we have an orbifold chart
$(U,V,\p,\G,\o_0)$ 
such that $U=B_\e(0) \subset \CC^2$ is a ball and $\p^{-1}(\{x\})=\{0\}$,
and $\G=\G_x < \U(2)$ acts in $U$ by unitary matrices.
Unless otherwise stated, 
from now on we assume that every orbifold chart of $(X,\o)$ has the form above. 
We will write $(U,V,\p,\G,\o_0)$ if moreover the symplectic form is standard in the chart,
and analogously for another tensors like $g_0$ and $\frj$.

In dimension $4$ the isotropy set can be expressed as a union $\S=\S^0\cup \S^* \cup \S^1$
of three subsets with distinct properties. These are determined by a geometric condition
that depends on the action of the isotropy groups $\G_x < \U(2)$ in $\CC^2$, as follows.

\begin{enumerate}
\item[]\textbf{Case 1:}  $x\in \S^0$ if the action of $\G_x$ on $\CC^2-\{0\}$ is free.
\medskip
\item[]\textbf{Case 2:}  $x\in \S^*$ if there exists a complex line $L \subset \CC^2$
such that for every $\g \in \G_x$ we have $\fix(\g)=D$.
\medskip
\item[]\textbf{Case 3:}  $x\in \S^1$ if there exist at least two complex lines $L_1, L_2 \subset \CC^2$
and non-identity elements $\g_1, \g_2 \in \G_x$ so that $L_1 \subset \fix(\g_1)$ and $L_2 \subset \fix(\g_2)$.
\end{enumerate}

Note the following:
\begin{itemize}
\item If $x\in \S^0$, then $x$ is an isolated point of $\S$.
That is why the points of $\S^0$ are called \emph{isolated singular points}.
\medskip
\item If $x \in \S^*$ then $D=\p(L)$ 
is contained on $\S^*$ and every point on this line has constant isotropy $\G_x$.
The connected components of $\S^*$
are therefore surfaces $S_i$ such that all its points have the same isotropy group $\G_i$.
\medskip
\item The points of $\S^1$ are also isolated; in addition these lie on the closure of some surfaces $S_i \subset \S^*$. 
Given $x\in \S^1$, let us call $I_x$ the set of indices $i$ such that the surface $S_i$ accumulates to $x$ and write $\G_i$ the isotropy set of $S_i$. 
\end{itemize}

We have the following result:

\begin{lemma} \label{lemma:trivial}
Let $p \in \S^1$ and let $(U,V,\G)$ be an orbifold chart around $p$, with $\G=\G_p < \U(2)$.
Let $\G^*= \la \G_i \mbox{ s.t. } i \in I_p \ra \lhd \G$ be the normal subgroup
generated by the isotropy groups of all the surfaces $S_i$ accumulating at
$p$. 
\begin{enumerate}
\item The space $U'=U/\G^*$ is a topological manifold and inherits naturally a complex orbifold structure with isotropy the surfaces $S_i$.

\item The quotient group $\G'=\G/\G^*$ has an induced action on $U'=U/\G^*$, 
and moreover $U'/\G'=U/\G$.
\end{enumerate}
\end{lemma}
\begin{proof}
We check first that $\G^*$ is a normal subgroup of $\G$.
Take $g_i \in \G_i$, and $\g \in \G$. Then $\g g_i \g^{-1}$ leaves fixed all the points
in the surface $\g(S_i) \subset U$. Hence $\g g_i \g^{-1}$ belongs to the isotropy group 
of some of the surfaces $S_j=\g(S_i)$. 
This means that $\G \circ \G_i \circ \G^{-1} \subset \bigcup_j \G_j \subset \G^*$.
If we take now a generic element of $\G^*$, 
i.e. finite product $\prod_k g_{i_k}$ with $g_{i_k} \in \G_{i_k}$,
then for any $\g \in \G$ we have $\g (\prod_k g_{i_k}) \g^{-1}= \prod_k (\g g_{i_k} \g^{-1}) \in \G^*$,
and this proves that $\G \circ \G^* \circ \G^{-1} \subset \G^*$, so $\G^*$ is a normal subgroup.

The complex orbifold structure on $U/\G^*$ exists because 
$\G$ acts in $U \subset \CC^2$ by biholormorphisms, so it acts $\frj$-equivariantly.
To see that $U'=U/\G^*$ is a topological manifold, observe that
the group $\G^*$ acts in $\CC^2$ and it is generated by complex reflections. 
Hence the algebra $\CC[z_1,z_2]^{\G^*}$ of $\G^*$-invariant polynomials is a polynomial algebra 
generated by $2$ elements, say $f,g$. This is proved for real reflections in \cite{Chev},
but the proof carries over to complex reflections also, see \cite{Prill}.
Consider 
\begin{equation*} 
H\colon \CC^2 \to \CC^2 \quad , \quad H(z)=(f(z_1,z_2),g(z_1,z_2)).
\end{equation*}
This map induces an homeomorphism $\bar{H} \colon \CC^2/\G^* \to \CC^2$.
That ensures that $U/\G^*$ is a topological manifold.

Now consider $\G'=\G/\G^*=\{\g \G^* \mbox{ s.t. } \g \in \G\}$, and define its action on 
$U'=U/\G^*=\{\G^* u \mbox{ s.t. } u \in U\}$ by $(\g \G^*) \cdot (\G^* u) = \G^* (\g u)$
for $u \in U$ and $\g \in \G$. This is well defined
since for $\g'=\g \g_1^*$ and $u'=\g_2^* u$ other representatives of $\g \G^*$ and $\G^* u$ we have $\g' u'= (\g \g_1^*) (\g_2^* u)=\g (\g_1^* \g_2^*) u = \g \g^* u = c_g(\g^*) \g  u$, where $c_\g \colon \G \to \G$ is conjugation by $\g$ maps $\G^*$ to itself,
and $\g^*=\g_1^* \g_2^* \in \G^*$ so $\G^* (\g' u')= \G^* (\g u)$. 
It is immediate to check that this gives an action.
Moreover, the orbit of $\G^* u$ in $U'/\G'$ is given by 
$\G' \cdot (\G^* u)=\{ \G^* (\g u) | \g \in \G \}$ 
so it equals $\G u$ the orbit of $u$ in $U/\G$.
\end{proof}

The following Lemma proves the existence of a suitable orbifold almost K{\"a}hler structure
in dimension $4$. It gives a local K{\"a}hler model around any point if $\S^1 \cup \S^0$.

\begin{lemma} \label{lem:orbi-almost-kahler}
Let $(X,\o)$ be a $4$-dimensional symplectic orbifold.
There exists an almost K{\"a}hler structure $(X,\o,g,J)$ such that:
\begin{enumerate}
\item For each point $p \in \S^0$, there is an orbifold chart
$(U,V,\G,\o_0,g_0,\frj)$ around $p$.

\item For each point $p \in \S^1$ there is an orbifold chart
$(U,V,\phi,\G,\o_0,g_0,\frj)$, and each surface $S_i$ that accumulates to $p$
lifts to $\phi^{-1}(S_i)$ which is a union of disjoint complex curves
in the chart $(U,\frj)$.
\end{enumerate}
\end{lemma}
\begin{proof}
We use Corollary \ref{cor:flat chart} to put an almost K{\"a}hler structure $(g,J)$
so that there are flat K{\"a}hler charts around any point in $\S^1 \cup \S^0$.
Now, using Corollary \ref{cor:overSH}, both statements are clear.
\end{proof}

\subsection{Tubular neighbourhood of singular surfaces}

With respect to the orbifold almost K{\"a}hler structure of above, 
given a surface $S \subset \S^*$, note that $TS^{\perp \o}=TS^{\perp g}$, i.e. for every $z \in S$,
the symplectic and metric orthogonal spaces to $T_zS$ are the same.
The following Lemma gives an orbifold atlas of $X$ such that a tubular neighborhood of any surface $S \subset \S^*$ inherits an
atlas of an orbifold disc-bundle with structure group in $\U(1)$.

\begin{lemma} \label{lem:atlas bundle}
The symplectic orbifold $(X,\o)$ admits an atlas $\cA$ such that
for any $S \subset \S^*$, some neighborhood $D_{\e_0}(\bar S)$ of $\bar{S}$ in $X$ 
admits an open cover $D_{\e_0}(\bar S)=\cup_\a V_\a$ such that 
for each $\a$ there is an orbifold chart
$(U_\a,V_\a,\G_\a,\phi_\a,\o_\a) \in \cA$, satisfying:
\begin{enumerate}
\item If $V_\a \cap \S^1= \emptyset$, then $U_\a=S_\a \x B_{\e_0}$ is a product, 
with $S_\a \subset S$ open, $D_{\e_0} \subset \CC$ a disc, and the group
$\G_\a=\G$ is the isotropy group of the surface $S$.
For any other $V_\b$ with $V_\b \cap \S^1= \emptyset$, the orbifold change of charts are given by
$$
\psi_{\a \b}=(\psi^1_{\a \b},\psi^2_{\a \b}) \colon U_\a \to U_\b \quad , \quad (z,w) \mapsto (\psi^1_{\a \b}(z),A_{\a \b}(z)w)
$$
with 
$$
A_{\a \b} \colon S_\a \to \U(1) \quad , \quad z \mapsto A_{\a \b}(z)
$$ 
a smooth function taking values in the unit circle $\U(1)$. 
The group $\G < \U(1)$ acts in $U_\a$ and $U_\b$ by a rotation in $D_{\e_0}$,
in particular it is isomorphic to $\ZZ_m$.

\item For each $p \in \S^1 \cap \bar{S}$ denote $V_p$ 
an open set of the cover that contains $p$.
Then the corresponding chart $(U_p,V_p,\G_p,\phi_p,\o_0)$ satisfies that 
$H_p \x D_{\e_0} \subset U_p$ with $\phi_p(H_p \x \{0\}) = \bar{S} \cap V^p$,
and if $V_\a$ does not contain $p$ the change of charts is given by
$$
\psi_{\a p} \colon U_\a \to U_p \quad , \quad (z,w) \mapsto (\psi^1_{\a p}(z),A_{\a p}(z)w)
$$
with $A_{\a p}(z) \in \U(1)$, and its image is 
$\psi_{\a p}(U_\a)=H_\a \x D_{\e_0} \subset H_p \x D_{\e_0}$, 
with $\phi_p(H_\a \x \{0\}) = S \cap V_\a \cap V_p$.
If we denote $\rho_{\a p} \colon \G_\a = \G \hookrightarrow \G_p$ the associated monomorphism
of isotropy groups, then the subgroup $\rho_{\a p}(\G) < \G_p$ 
acts on $H_\a \x D_{\e_0}$ as a rotation in $D_{\e_0}$.

\end{enumerate}
\end{lemma}
\begin{proof}
Consider an orbifold almost K{\"a}hler structure $(\o,g,J)$ on $X$ 
as in Lemma \ref{lem:orbi-almost-kahler}.
To see (1), take an initial cover $\cup_\a V_\a$ of $S$ 
with orbifold charts $(U'_\a, V_\a,\G, \o_0,g_\a,J_\a)$ such that $V_\a \cap \S^1=\emptyset$. 
Let $(z,w)$ be coordinates in $U'_\a$, such that $S_\a=S \cap U'_\a=\{w=0\}$.
Recall that we have for $z=(z,0) \in S_\a$ an identification $(T_z S_\a)^{\perp}=\{z\} \x \CC$.
The change of charts are given by
\begin{align*}
\f_{\a \b} \colon \, & U'_\a \to U'_\b \\ 
            & (z,w) \mapsto (\f^1_{\a \b}(z,w),\f^2_{\a \b}(z,w)=(z',w'))
\end{align*}
with $\f^2_{\a \b}(z,0)=0$ for all $(z,0) \in S_\a$.
Consider now $U_\a=S_\a \x \CC$, and the maps
\begin{align*}
\phi_{\a \b} \colon \, & U_\a=S_\a \x \CC \to U_\b=S_\b \x \CC \\
               & (z,u) \mapsto (\phi^1_{\a \b}(z),A'_{\a \b}(z)u)=(z',u')
\end{align*}
with $\phi^1_{\a \b}(z)=\f^1_{\a \b}(z,0)$, and $A'_{\a \b}(z)=\bd_w \f^2_{\a \b}|_{(z,0)}$.
Here $\bd_w \f^2_{\a \b}$ stands for the Jacobian matrix of $\f^2_{\a \b}$ in the variable $w$.
Now we use the exponential map to identify $U'_\a$ and $U_\a=S_\a \x D_\e$, where $D_\e \subset \CC$
is a small disc. To this end let us consider the maps 
$$
e_\a \colon U_\a=S_\a \x D_\e \to U'_\a \quad , \quad (z,u) \mapsto \exp_z(u)=(z,w)
$$
which are diffeomorphisms $\e \leq \e_0$, maybe reducing $U'_\a$.
The induced action of the group $\G$ in $U_\a=S_\a \x D_{\e}$ 
is given by complex multiplication in $D_{\e}$.
Now, it is easy to check that the maps $\p_{\a \b}$ are the induced change of charts 
with respect to the new coordinates $(z,u)$ and $(z',u')$ in $U_\a$ and $U_\b$.
In other words, $\phi_{\a \b} =e_\b^{-1} \circ \f_{\a \b} \circ e_\a.$
Hence we can take the maps $\p_{\a \b}$ as new orbifold change of charts.
The matrices 
$$
A'_{\a \b}(z) \colon ((T_zS_\a)^{\perp}, h_\a|) \to ((T_{z'}S_\b)^{\perp}, h_\b|)
$$
are isometries with respect to the orbifold hermitian metrics 
$h_\a=g_\a + \o_0(\cdot, J_\a \cdot)$ and $h_\b=g_\b + \o_0(\cdot, J_\b \cdot)$
restricted to the orthogonal spaces to $S$ (we use the notation $h_\a|$ to express this restriction). 
In particular $A'_{\a \b}(z) \in \Sp(2)$ are symplectic matrices.
Take orthonormal basis of $((T_zS_\a)^{\perp}, h_\a|)$
and $((T_{z'}S_\b)^{\perp}, h_\b|)$ so that $h_\a|_z$ and $h_\b|_z$
become the standard hermitian metric $h_0$,
and denote $P_\a(z), P_\b(z') \in \Sp(2,\RR)$ the matrices of change of basis. 
Call the new coordinates $(z,v)=(z,P_\a(z)u)$
and $(z',v')=(z',P_\b(z')u')$.
The change of trivializations in the new coordinates are given by
the matrices
$$
A''_{\a \b}(z)=P_\b(z') \cdot A'_{\a \b}(z) \cdot (P_\a(z))^{-1} \in \U(1).
$$
These matrices are unitary as we want, but the isotropy groups act via 
$$
\G_z=P_\a(z)\cdot \G \cdot P_\a(z)^{-1} \quad , 
\quad \G_{z'}=P_\b(z')\cdot \G \cdot P_\b(z')^{-1}
$$
so they are groups acting non-linearly. To fix this, 
consider $r \colon \Sp(2,\RR) \to \U(1)$ the retraction given in (\ref{eqn:r(A)}).
By Lemma \ref{lem:retrac} we have
$$
\G_z=r(P_\a(z))\cdot \G \cdot r(P_\a(z))^{-1} \quad , 
\quad \G_{z'}=r(P_\b(z')) \cdot \G \cdot r(P_\b(z'))^{-1}.
$$
So if we introduce another coordinates $w$ in $U'_\a$ by $(z,w)=(z,r(P_\a(z))^{-1}v)$ 
and $w'$ in $U'_\b$ by $(z',w')=(z',r(P_\b(z'))^{-1}v')$
then the corresponding transition matrices are given by
$$
A_{\a \b}(z)=r(P_\b(z'))^{-1} \cdot A''_{\a \b}(z) \cdot r(P_\a(z)) \in \U(1)
$$
and moreover the varying groups $\G_z$ and $\G_{z'}$
become $\G$ again. This shows what we wanted.
The sought transition maps $\psi_{\a \b}$ are given by $\psi^1_{\a \b}(z)=\phi^1_{\a \b}(z)$ 
and $\psi^2_{\a \b}(z,w)=A_{\a \b}(z)w$

Now let us see (2). Suppose that $S$ accumulates at $p\in \S^1$, and let $(U,V,\p,\G)$ 
a chart around $p=\p(0)$ with coordinates $(z,w)$ 
such that $(\o,g,J)$ is the standard K{\"a}hler structure in this chart. 
After a complex rotation on $U$ (which preserves the whole structure) 
we can suppose that $\bar{S} \cap V = \p(\{w=0\})$.
In this case, $e_U(z,w)=(z,w)$ so that $(U,V,\p,\G)$ remains invariant 
after the process described before.
\end{proof}

\begin{remark} \label{rem:rotation}
The proof of this Lemma shows that,
given the K{\"a}hler chart $\p \colon U_p \to V_p$ of a point $p \in \S^1$,
the atlas for the tubular neighborhood $D_{\e_0}(S)$ of a singular surface $S$ with $p\in \overline{S}$ can be constructed making a complex
rotation of the preimage $\p^{-1}(V_p \cap D_{\e_0}(S))$ so that $S=\{w=0\}$.
\end{remark}

\begin{remark} \label{rem:restriction}
Near $p \in \S^1 \cap \bar{S}$ one can define a \emph{compatible orbifold chart} from $(U_p,V_p,\G_p,\p_p)$: 
we let $\e_p>0$ be such that  $B_{3\e_p}(p)\subset V_p$ and let $\e_0>0$ such that 
$$
\p_p ((B_{3 \e_p}(0) -  B_{\e_p}(0)) \cap (\CC \x D_{\e_0})) \subset X- \cup_{S' \neq S} D_{\e_0}(\bar{S'}).
$$
There is a compatible orbifold chart $(A_p^3,V_{A_p^3},\tilde \G, \p_p)$ 
with $A^3_p=(B_{3 \e_p}(0) -  B_{ \e_p}(0)) \cap (\CC \x D_{\e_0})$, $V_{A_p^3}=\phi_p(A_p^3)$ 
and $\tilde \G=\{ \g \in \G_p \mbox{ s.t. } \g(z,0)=(z',0) \}< \U(1)\x \U(1)$.

Moreover, if $\G_S$ is the isotropy of $S$ then $A_p/\G_S \to A_p/\tilde \G$ is a covering with Deck group $\tilde{\G}/\G_S$. In addition, given $\p_p(z,0)\in U_{A_p^3}$ one can restrict sufficiently the previous chart to obtain an orbifold chart of $X$ with isotropy $\G_S$.
\end{remark}

\begin{remarks} \label{rem:almost standard}
\begin{enumerate}
\item The symplectic forms $\o_\a=e_\a^*\o_0$ of the atlas above
may not be standard in the charts $U'_\a=S_\a \x D_\e$,
but they are standard at the points of $S$, so we have in coordinates
$(z,w) \in U'_\a$ the expression
$$
\o_\a=-\tfrac{\ii}{2} (dz \wedge d \bar z + dw \wedge d \bar w) + O(|w|).
$$
\item The atlas $\cA$ constructed above can be refined so that for any $p \in \S^1$
and any neighborhood $W^p$ of $p$ in $X$, there is an orbifold chart $(U^p,V^p,\p_p,\G_p)$
in $\cA$ with $p \in V^p \subset W^p$. 
Also, we can assume that only one of the open sets of the atlas contains the point $p \in \S^1$.
\end{enumerate}
\end{remarks}

Consider an orbifold almost K{\"a}hler structure $(X,\o,g,J)$ of Lemma
\ref{lem:orbi-almost-kahler}.
Let $S \subset \S^*$ be an isotropy surface, and $D_{\e_0}(\bar S)$ a neighborhood 
of $\bar{S}$ in $X$ as in Lemma \ref{lem:atlas bundle}, 
with an open cover $D_{\e_0}(\bar S)=\cup_\a V_\a$ and orbifold charts $(U_\a,V_\a,\G_\a,\phi_\a,\o_\a)$.
For $p \in \bar{S} \cap \S^1$ let $(U_p,V_p,\G_p,\phi_p,\o_0)$ be the unique orbifold chart covering $p$.
Denote $\pi \colon D_{\e_0}(\bar S) \to \bar{S}$ the projection.
The following Lemma shows the existence of an orbifold connection $1$-form
in $D_{\e_0}(\bar S)- (\cup_{p\in \S^1\cap \bar{S}} B_{\e_p}(p) \cup \bar{S})$, where $\e_p$ verifies that $B_{3\e_p}(0)\subset U_p$ and $3\e_p<\e_0$.

\begin{lemma} \label{lem:connection}
Notations as above.
There exists an orbifold $1$-form $\eta=\eta_S \in \O^1_{orb}(D_{\e_0}(\bar S)- (\cup_{p\in \S^1\cap \bar{S}} B_\d(\e_p) \cup \bar{S}))$ such that:
\begin{enumerate}
\item If $V_\a \cap \S^1=\emptyset$, the liftings $\eta_\a$ 
in the orbifold charts $U_\a=S_\a \x D_\e$ have the form
$\eta_\a= d \h + \pi^* \nu_\a$ for $\nu_\a \in \O^1(S_\a)$, 
with $\h$ the angular coordinate in $D_{\e_0}$.

\item For $p \in \S^1 \cap \bar S$, let $H_p \x D_{\e_0} \subset U_p$ with $\phi(H_p \x \{0\})=\bar{S} \cap V_p$.
Then, the lifting of $\eta$ in $U_p-B_{\e_p}(p)$ equals $d \h$ in $V_{A^2_p}$, with $V_{A_p^2}=\phi_p(A^2_p)$ and $A^2_p=(B_{2 \e_p}(0) -  B_{ \e_p}(0)) \cap (\CC \x D_{\e_0'})$.
\end{enumerate}
\end{lemma}
\begin{proof}
Consider $\pi_\a \colon S_\a \x D_{\e_0} \to D_{\e_0}$, and the angular function $\pi_\a^*\h$ which measures the angle in each fiber $D_{\e_0}$. 
We have that $\pi_\a^* \h- \pi_\b^*\h=\pi^*\xi_{\a \b}$ in the intersections,
being $\xi_{\a \b}=\xi_{\a \b}(z)$ a function on $S$.

The $1$-forms 
$d \pi_\a^* \h - d \pi_\b^*\h=\pi^* d \xi_{\a \b}=\pi^* \nu_{\a \b}$
are $\G$-invariant since $\G$ acts on the angle $\h$ as a translation
in the charts. The argument carries also on the chart $(A_p^3,V_{A_p^3},\tilde{\G},\p_p)$ defined on Remark \ref{rem:restriction}; the angular form is $\tilde{\G}$-invariant because each element of $\tilde{G}$ can be expressed as the composition of a map $(z,w) \to (e^{\frac{2\pi i}{k}}z,w)$ which preserves the angle, 
and a map that acts on the angle $\h$ as a translation.

We take a cover of $\bar{S}- \cup_{p \in  \S^1 \cap \bar{S}}B_\d(\e_p)$ 
formed by coordinate open sets and such that all points of $V_{A_p^2}$ are covered only by $V_{A_p^3}$. 
We denote it by $\{V_\a\}_{\a\in\D}$.
Now, taking a partition of unity $\rho_\a$ subordinated to the cover $\{V_\a\}$
we can define $\eta_\a=\sum_\a \pi^* \rho_\a \cdot \pi_\a^* (d \theta)$.
If we fix a chart $U_\b=S_\b \x D_{\e_0}$, then the lifting of $\eta$ to $U_\b$ is given by 
$$
\eta|_{U_\b} = \sum_\a \pi^* \rho_\a \cdot (\pi_\b^* (d \h) + \pi^* \nu_{\a \b})
=\pi_\b^* (d \h) + \sum_\a \pi^* (\rho_\a \cdot \nu_{\a \b}).
$$
This proves that $\eta$ restricts to $d \h$ on each fiber and (1). 

Take $p \in \S^1$. Since points on $\p_p((B_{2 \e_p}(0) -  B_{ \e_p}(0)) \cap \CC \x D_{\e_0})$ 
are covered by a unique open set of the covering, the connection is trivial over it. This proves $(2)$.
\end{proof}

\section{Resolution.} \label{sec:resolution}
In this section we shall be explicit about the atlas that we consider in the space $X$. 
Let $\cA$ the orbifold atlas of $X$ that verifies Lemma \ref{lem:atlas bundle} 
and denote the symplectic orbifold structure by $(X,\o,\cA)$.
 
First observe that we can suppose that $\S^0= \emptyset$ because we can use the method described in \cite{CFM} to resolve isolated singularities; we briefly describe it in subsection \ref{ss:isolated-singularities}. In order to perform the resolution we first endow $X$ with the structure of a symplectic orbifold $(X,\widehat{\mathcal{A}},\widehat{\o})$ without changing the underlying topological manifold. The isotropy points of the new structure are isolated and consist of $\S^0 \cup \S^1$. For that purpose, we first construct a manifold atlas of $X-\S^1$ and replace $\o$ with a closed $2$-form $\o_a^*$, which is zero on a closed neighbourhood of $\S^1$ and symplectic out of it. After this we extend the orbifold structure to $\S^1$ to obtain the desired orbifold structure $(X,\widehat{\mathcal{A}})$. The orbifold form $\o'$ naturally extends to $(X,\widehat{\mathcal{A}})$; we finally use a gluing lemma (see Lemma \ref{lem:gluing_forms}) to contruct $\widehat{\o}$. 

The extension process is inspired in Lemma \ref{lemma:trivial}. 
Following its notation, if $p\in \S^1$ and $(U,V,\G,\p)$ is an orbifold chart then $V=U/\G= (U/\G^*)/\G'$. An holomorphic homeomorphism $H\colon U/\G^* \to \widehat U \subset \CC^2$ allows us to resolve the singularities of $\S^*\cap V$; and $\widehat U/\G'$ has an isolated singularity at $0$. This structure must be compatible with the atlas defined on $X-\S^1$; for that reason we resolve the singularities on $\S^*$ using complex transformations. Riemman extension theorem will ensure the compatibility of both structures out of $\S^1$.

We finally resolve the isolated isotropy locus of 
$(X,\widehat{\mathcal{A}},\widehat{\o})$ using again the method of \cite{CFM}. 
This process yields a resolution of $(X,\mathcal{A},\o)$ as follows:

\begin{theorem} \label{thm:main-thm-bis}
Let $(X, \o)$ be a symplectic $4$-orbifold such that 
the closure of each conected component $S\subset \S^*$ is  compact.
There exists a symplectic manifold $(\tilde{X}, \tilde{\o})$ and a smooth map 
$\pi \colon (\tilde{X},\tilde{\o}) \to (X,\o)$
which is a symplectomorphism outside an arbitrarily small neighborhood of the isotropy set of $X$.
\end{theorem}

\subsection{Resolution of isolated singularities}\label{ss:isolated-singularities}

We briefly outline the process of resolving an isolated singularity, which can be found in \cite{CFM}. As one should observe, this method is valid for symplectic orbifolds of arbitrary dimension; but we restrict to the case that the dimension is $4$.

Take an isolated singular point $p\in \S^0$ and a K{\"a}hler Darboux chart $(U,V,\G,\p,\o_0, \frj)$
around $p$, with $V \cong U/\G$, $\G <\U(2)$. 
The space $U/\G$ is an affine variety because one can consider $\la P_1,\dots, P_N \ra$ 
a basis of the finitely generated $\CC$-algebra of polynomials that are invariant by the action of $\G$, and define the holomorphic embedding:
$$
\iota \colon \CC^2 /\G \to \CC^N, \qquad \iota(x)=(P_1,\dots,P_N)(x).
$$
One can use the model $\iota(\CC^2/\G)$ to perform the resolution of singularities; 
this obtained by a finite number of blow-ups. Thus the resolution $b \colon F \to \iota(\CC^2/\G)$ is quasi-proyective and consequently K\" ahler. We shall denote by $\o_F$ the symplectic form on the resolution.

Then we replace $B_\e(p)=\p(B_\e(0)) \subset V$ by a small ball around the exceptional set
$E=b^{-1}(0)$ in $F$; that is, define:
$$
X'=(X-B_{\e}(p)) \cup_{\bar{\phi} \circ b} b^{-1}(B_\e(0)/\G),
$$
To endow $X'$ with a symplectic form we interpolate $b^*\o_0$ and $\l \o_F$ on $A=b^{-1}(B_{3\d}(0)-B_{\d}(0)/\G)$, where $\l$ is small enough. The interpolation is allowed due to the fact that $A$ is a lens space and thus $H^2(A,\RR)=0$; in order to do so one has to replace the K\"ahler potential $r^2$ of $\o_0$ with a radial K\" ahler potential on $\CC^2- B_{\d}(0)$ that vanishes on $\overline{B_{\d}(0)}$ and coincides with $r^2$ on $\CC^2-B_{2\d}(0)$, obtaining a form $\o_1$.
If $d\eta= \o_F - b^*\o_1$ on $A$, one can ensure that if $\l$ is small enough and $\rho$ is a radial bump function which is $1$ on $B_{2\d}(0)$ and $0$ on $\CC^2-B_{3\d}(0)$ , that:
$$
\o_\l= b^*\o_1 + \l d( (\rho \circ b) \eta)
$$
extends to a symplectic form on $b^{-1}(B_{2\e}(0)/\G)$ and interpolates the desired forms; this is simmilar to the gluing process described on Lemma \ref{lem:gluing_forms}. To ensure that $\o_\l$ is symplectic on $\d < r \leq 2\d$ we use the fact that both $b^*\o_1$ and $\o_F$ are positive with respect to the complex structure $\frj_F$ on $F$.

\subsection{Construction of $(X-\S^1, \widehat{\mathcal{A}}, \o')$.}

In this first step we resolve each surface $S\subset \S^*$ separately; working away of $\S^1$. We split the construction in two parts: we first do a preparation on the orbifold $(X,\o, \cA)$ and then change the symplectic orbifold structure.

\subsection*{Preparation}

In order to construct a smooth atlas $\widehat \cA$ of $X-\S^1$ we shall modify $\cA$ around singular surfaces. For this, we use the basic fact that the map $q \colon \CC \to \CC$, $q(z)=z^m$ gives a homeomorphism between $\CC/\ZZ_m$ and $\CC$.
This map applied to the fibers $\{z\} \x D_{\e_0} \subset D_{\e_0}(\bar S)$ yields a manifold atlas of $D_{\e_0}(\bar S)-\S^1\cap \bar{S}$,
hence providing the sought manifold atlas $\widehat{\cA}$ on $X-\S^1$.

But the symplectic form $\o$ is singular on $\S^*$ with respect to the atlas $\widehat \cA$ of $X-\S^1$.  For this reason we replace $\o$ on the orbifold $(X,\cA)$ with a form $\o_a^*$ that is degenerated on each $S\subset \S^*$, but it will be symplectic on the manifold $(X- \overline{B_\e(\S^1)}, \widehat{\cA})$. Here $B(\S^1)$
stands for a neighborhood of $\S^1$ which is a union of balls around each $p\in \S^1$ that are contained in $V_p$, where $(U_p,V_p,\G_p,\o_0)\in \cA$ is the Darboux chart as usual. More precisely, given $p\in \S^1$ the ball is $\p(B_{\e_p}(0))$ where $\e_p>0$ verifies $B_{3\e_p}(0)\subset U_p$. In addition, $\o_a^*=0$ on $\overline{B(\S^1)}$. 

As a first step, we need an orbifold symplectic form $\o^0$ in $X$ which is constant in the fibers of $D_{\e_0}(\bar S)$ for each $S \subset \S^*$. For that purpose we first introduce some notations; let $S$ be an isotropy surface, we denote 
$\pi \colon D_{\e_0}(\bar S) \to \bar S$ the projection. By Lemma \ref{lem:connection}
we have an orbifold connection $1$-form $\eta$ on $D_{\e_0}(\bar S)-\bar S \x \{0\}$ 
which equals $d \h$ in each punctured fiber $\{z\} \x (D_\e - \{0\})$, $z \in \bar S$.
Denote $\o_S= \iota^* \o \in \O^2(S)$ the symplectic form in $S$, 
with $\iota \colon \bar S \hookrightarrow D_{\e_0}(\bar S) \subset X$ the inclusion.

\begin{lemma} \label{lem:omega^0}
For any choice of $\d>0$ small enough, there exists an orbifold symplectic 
form $\o^0=\o^0(\d)$ on $X$ such that $\o^0=\o$ in 
$$
(X- \cup_{S \subset \S^*} D_{2\d}(\bar S))\cup_{p \in \S^1} V_p,
$$
and for every singular surface $S\subset \S^*$,
$\o^0=\pi^* \o_S +  r dr \wedge \eta  + \tfrac{1}{2} r^2 d \eta$ in
$D_{\d}(\bar S)$; where $\pi \colon D_{\d}(\bar S) \to \bar S$ denotes the projection.
\end{lemma}
\begin{proof}
Let $S\subset  \S^*$ be a singular surface; we define an orbifold $2$-form by
$$
\o'=\pi^* \o_S +  r dr \wedge \eta  + \tfrac{1}{2} r^2 d \eta \in \O^2_{orb}(D_{\e_0}(\bar S))
$$
where $r$ is the function in $D_{\e_0}(\bar S)$ measuring the radius of the fiber $D_{\e_0}$.
A simple computation shows that $\o'$ is smooth for $r=0$, and that $d \o'=0$. In addition, given $p \in \S^1 \cap \overline{S}$, it holds that $\eta=d\theta$ and $\o$ is the standard  K{\"a}hler form on the set $H_p\times D_\e$; therefore $\o'$ coincides with $\o$. 

It is clear that $\o'$ is non-degenerate
at every point of the zero section $\bar S \x \{0\}$, so it is non-degenerate in a maybe smaller neighborhood which we call again $D_{\e_0}(\bar S)$.
Now we interpolate $\o'$ and $\o$ to obtain the sought orbifold symplectic form $\o^0$ on $X$.
Since $\iota^*(\o'-\o)=0$ and $D_{\e_0}(\bar S)$ retracts onto $\bar S$, 
we have that $\o'-\o=d \b$ for some orbifold $1$-form $\b$ defined in $D_{\e_0}(\bar S)$ which is $0$ on $H_p\times D_\e$.
By remark \ref{rem:almost standard} we have $|\o'-\o| = O(r)$ in $D_{\e_0}(\bar S)$.
We can take a primitive $\b$ of $\o'-\o=d \b$ such that $|\b| = O(r^2)$. 
Indeed, we can write $\o'-\o=\a_0 \wedge dr + \a_1$ for $\a_0$ a $1$-form
and $\a_1$ a $2$-form with $\a_1(\bd_r,\cdot)=0$.
Then we set $\b= \int_{0}^r \a_0 dr$, which is smooth and a primitive for $\o'-\o$
such that 
$$
|\b| \le C r |\a_0| = C r |(\o'-\o)(\bd_r,\cdot)| \le C r |\o-\o'||\bd_r|=O(r^2)
$$
since $|\bd_r|$ is bounded.
Now consider a bump function $\rho_\d(r)$ which equals $1$ in $D_\d(\bar S)$ 
and $0$ in $D_{2 \d}(\bar S)$ and such that $|\rho'_\d| \le \tfrac{3}{\d}$.  
Here $\d < \tfrac{\e_0}{2}$ is small, to be fixed later.
Define $\o^0= \o + d(\rho_\d \b)$.
We have that 
$$
|\o^0 - \o|=|\rho'_\d(r) dr \wedge \b + \rho_\d d \b| = O(\tfrac{r^2}{\d}) + O(r)
$$
so $\o^0 - \o=O(\d)$ in $D_{2 \d}(S)$. Hence $\o^0$ is symplectic in $D_{2 \d}(\bar S)$ for $\d$ small. 
Outside $D_{2 \d}(S)$ we have $\o^0=\o$ 
so it is also symplectic, and then $\o^0$ is a global orbifold symplectic form in $X$.
Note that $\o^0$ equals $\o'$ in $D_{\d}(S)$, as desired. Finally, 
it is clear from the construction of $\o^0$ that $\o=\o^0$ on a neighborhood $\cup_p V_p$ of $\S^1$.
\end{proof}

We now modify $\o'$ in order to obtain an intermediate form, $\o_a \in \O^2(X-B(\S^1))$. This is a closed form which is symplectic on $X-(\S^* \cup B(\S^1))$ but $\o_a \neq 0$ on $\partial B(\S^*)$; we contruct later the desired form $\o_a^*$ from $\o_a$. The construction of $\o_a$ follows the ideas of the proof of Lemma \ref{lem:omega^0} and consists on defining a symplectic form which is adapted to a splitting of the tangent bundle of each singular surface $S\subset \S^*$ into two distributions that we now introduce. 

Recall that our atlas provides a well-defined radial function around $S$. The connection $1$-form $\eta$ defined on Lemma \ref{lem:connection} allows us to define the horizontal subbundle $\cH= \ker(rdr\wedge\eta)$ and the vertical subbundle $\cV=\ker(d\pi)$; these can be endowed with an almost K{\"a}hler structure:

\begin{enumerate}
\item On the horizontal space we consider the symplectic form $\pi^*\omega_S$; if $J_S$ tames $\o_S$ on $S$, we can extend it to $\cH$ via the isomorphism $d\pi$ and continue denoting it with the same name. This extension tames $\pi^*\o_S$ because $(d\pi)^t(\o_S)=\pi^*\omega_S$.

\item On the vertical bundle $\cV$ we consider the standard metric $g_{\cV}=dr^2+r^2 d \h^2$ and the complex structure $J_{\cV}$ induced by the complex multiplication by $\ii$ in the atlas $(U_\a,V_\a,\G_\a,\p_\a)$. The induced form is
$\o_{\cV}=rdr\wedge d\theta = rdr \wedge \eta|_{\cV}$.
\end{enumerate}
Note that ${\cH}^* \cong \mathrm{Ann}({\cV})= \mathcal{C}^\infty \otimes \pi^*(\Omega^1(S))$
and that ${\cV}^* \cong \mathrm{Ann}({\cH})=\mathcal{C}^\infty \otimes \la dr, \eta \ra$,
so we can extend any tensor initially constructed in the horizontal (vertical) distribution
as being zero in the vertical (horizontal) distribution respectively. This applies
especially to $J_{\cV}$.

Before stating the result in which we construct the form $\o_a$, we introduce some notations. Consider the neighborhoods $D_\d(\bar S)$ for $0 < \d \le \e_0$; there exists $\d_p^{\S^*}>0$ such that
for any $0 < \d < \d_p^{\S^*}$ it holds 
$D_\d(\bar S) \cap D_\d(\bar S') \subset B(\S^1)$ for any pair of singular surfaces $S, S'$. Fixed a singular surface $S$, define for $0 < \d < \d_p^{\S^*}$ the $\d$-normal neighborhood of $S-B(\S^1)$
$$
N_{\d}(S)= \bigcup_{\a \in \L_S} \p_\a(S_\a \times B_{\d}(0)),
$$ 
where $\L_S$ denotes the set of indexes $\a$ such that
$V_\a \cap S \ne \emptyset$ and $V_\a \subset X- B(\S^1)$. To ease notation, we assume that $\e_0$ is chosen so that 
$\e_0 < \d_p^{\S^*}$. Hence the neighborhoods $N_{\e_0}(S)$
are disjoint for the different surfaces $S$.

\begin{proposition}\label{prop:omega-a}
For every isotropy surface $S$ there exist $0<\d_0^S<\frac{1}{2}\d_p^{\S^*}$, $\d_2^S < \frac{1}{3}\d_0^S$, and $a_2^S>0$ such that for every $a< \min_S \{a_2^S\}$ there is a closed form $\o_a \in \O^2(X-B(\S^1))$ which is non-degenerate on $X- \left(  B(\S^1) \cup \S^* \right)$, such that $\o_a = \o$ on $X-\left( \cup_{p \in \S^1}B_{\e_p}(p) \cup_{S\in \S^*} N_{2\d_0^S}(S) \right)$, and for each singular surface we have:
$$
\o_a= \pi^*\o_S - \frac{1}{4} d J_{\cV} d(r^{2m} + a^2)^{\frac{1}{m}}
$$
on $N_{\d_2^S}(S)$. Here $\eta_S$ is the connection $1$-form constructed on Lemma \ref{lem:connection}.
On $\cup_{p\in \S^1} ( B_{2 \e_p}(p)-B_{\e_p}(p))-\S^*$ the form $\o_a$ is $\frj$-tamed and K{\"a}hler.
\end{proposition}
\begin{proof} 
We describe the process in a neighborhood of a fixed singular surface $S$.
Note that $J_{\cV}(dr)=-r \eta$ for $r\neq 0$, so in particular
$$
\frac{1}{2} d(r^2\eta)= -\frac{1}{2} d(rJ_{\cV}dr)=  -\frac{1}{4} dJ_{\cV}dr^2.
$$
Let $\o^0$ be the symplectic form of Lemma \ref{lem:omega^0}
such that 
$$
\o^0 = \pi^*(\o_S) + \frac{1}{2} d(r^2\eta)=  \pi^*(\o_S) - \frac{1}{4} dJ_{\cV}dr^2
$$
on $N_{\d_0}(S)$ and $\o^0=\o$ on $X-N_{2\d_0}(S)$, 
for some $\d_0$ with $0<\d_0< \tfrac{1}{2} \d^{\S^*}_p$.

Define the $2$-form: 
$$
\o^0_{a}  = \pi^*(\o_S) - \frac{1}{4}  dJ_{\cV}d(f(r^2,a)),
$$
where 
$ 
f(r,a)=(r^{m}+ a^2)^{\frac{1}{m}}.
$

Given a function $\bar{f}\colon \RR \to \RR$, the $2$-form 
$-\tfrac{1}{4} dJ_{\cV}d(\bar{f}(r^2))$ is expressed as follows:
\begin{align*}
-\tfrac{1}{4}  dJ_{\cV}d\bar{f}(r^2) =& - \tfrac{1}{2} dJ_{\cV}(r\bar{f}'(r^2)dr)= \tfrac{1}{2} d(\bar{f}'(r^2)r^2\eta)\\ 
=& \tfrac{1}{2}  r^2\bar{f}'(r^2)\pi^* \k + ( r^2 \bar{f}''(r^2) + \bar{f}'(r^2) ) rdr\wedge \eta, 
\end{align*}
where $\pi^*(\k)=d\eta$ is the curvature of the connection. In addition, we observe:
\begin{enumerate}
\item The projection of $-\tfrac{1}{4} dJ_{\cV}d\bar{f}$ to the space $\L^2 {\cV}^*$ is 
$$
-\tfrac{1}{4} dJ_{\cV}d\bar{f}|_{\cV}=( r^2 \bar{f}''(r^2) + \bar{f}'(r^2) ) rdr\wedge \eta.
$$
It is $J_{\cV}$-tamed on an annulus $R_0 \leq r \leq R_1$
as long as $ x \bar{f}''(x) + \bar{f}'(x) >0$ for $x \in [R_0^2,R_1^2]$.

\medskip

\item Denote $\| \cdot \|$ the norm with respect to the metric $g_S + g_{\cV}$. If $r\leq 1$ then,
\begin{align*}
\| \tfrac{1}{4}  dJ_v d\bar{f} \| \leq & \tfrac{1}{2} |\bar{f}'(r^2)| \|\pi^*\k\| + |\bar{f}''(r^2)| + |\bar{f}'(r^2))|. 
\end{align*} 
In particular, if $\Delta= [\bar{\d}_1,\bar{\d}_2]\subset [0,1]$ and $\bar{f}_a$ is a family of functions such that $\bar{f}_a|_{[\bar{\d}_1^2,\bar{\d}_2^2]}$ tends uniformly to $0$ as $a \to 0$ in the $C^2$ norm, then given $\e>0$ one can choose $a_0>0$ small enough such that for $a<a_0$,
$
\| \frac{1}{4}  dJ_v d\bar{f}_a \| < \e
$
on $r\in \Delta$.

\end{enumerate}

We now check that we can choose $\d_1<\tfrac{1}{2} \d_0$ and $a_1>0$ such that for every $a<a_1$ the form $\o^0_{a}$ is non-degenerate on $0<r<\d_1$. The vertical part $ \o^0_a|_{\cV}= -\tfrac{1}{4} dJ_{\cV}df(r^2,a)|_{\cV} $ is non-degenerate and $J_{\cV}$-tamed on $r\neq 0$ because:
\begin{align*}
\left(\frac{d}{dr}f\right)(r,a)=& r^{m-1}(r^{m} + a^2)^{\frac{1}{m}-1} >0 \\
\left(\frac{d}{dr^2}f\right)(r,a)=& a^2(m-1) r^{m-2}(r^{m} +a^2)^{\frac{1}{m}-2} >0 \, .
\end{align*}
The horizontal part is $\o^0_a|_{\cH} = \pi^*(\o_S) + \frac{1}{2} r^2\left( \frac{d}{dr}f \right) (r^2,a)\pi^*\k $,
whose first summand $\pi^*(\o_S)$ is non-degenerate and $J_S$-tamed on ${\cH}$; since $r^2\left( \frac{d}{dr}f \right) (r^2,0)=r^2$ we conclude the existence of $\d_1<\tfrac{1}{2} \d_0$ and $a_1>0$ such that ${\o}^0_a|_{\cH}$ is non-degenerate and $J_S$-tamed on ${\cH}$ for $r<\d_1$ and $a<a_1$.

Choose $\d_2<\frac{1}{3} \d_1$; we now show that there exists $a_2<a_1$ such that for every $a<a_2$ there is a form
$\o_{a}$ on $X$ with $\o_{a}=\o^0_{a}$ if $0\leq r \leq \d_2$, $\o_{a}=\o^0$ if $r>2\d_2$ and
such that $\o_a$ is $J_{\cV} + J_{\cH}$ tamed on $\d_2<r<2\d_2$. 
Let $\rho=\rho(x)$ be a smooth function such that $\rho=1$ if $x\leq 1$ and $\rho=0$ if $x\geq 4$ and define $\rho_\d(x)=\rho(\frac{x}{\d^2})$. We also define $h(x,a)= f(x,a)- x$, $H(x,a)=\rho_{\d_2}(x)h(x,a)$ and the closed form
$$
\o_{a}= \o^0 - \frac{1}{4}dJ_{\cV}d (H(r^2,a)).
$$
We now show that this is $J_{\cV} + J_{\cH}$ tamed on $\d_2<r<2\d_2$.
Note that the function $H$ is smooth on $(x,a)\in (0,\infty)\times \RR$ and verifies that $H(x,0)= \rho_{\d_2}(x) h(x,0)=0$. Thus, the family $\bar{f}_a(x)=H(x,a)$ converges uniformly to $0$ in the $C^2$ norm on the domain $x\in [\d_2^2,4\d_2^2]$.

Let $\e>0$ be such that an $\e$-ball with respect to $g_S + g_{\cV}$
around $\o^0$ is $(J_S + J_{\cV})$-tamed on $\d_2 \leq r \leq 2\d_2$.
Our previous observation ensures the existence of $a_2>0$ such that for every $a<a_2$:
$$
\|  \o_{a}- \o^0\|= \| \tfrac{1}{4} dJ_{\cV}d (H_\d (r^2,a)) \| < \e
$$
on $\d_2 \leq r \leq 2\d_2$, and thus $\o_{a}$ is $J_{\cV} + J_{\cH}$-tamed on $\d_2 \leq r \leq 2\d_2$,
so it is a symplectic form there.

Note also that on the chart $B_{2\e_p}(p) \subset V_p$ the connection is flat, i.e. $\eta=d \h$, and moreover $(\o_S,J_S)$ becomes the standard K{\"a}hler structure on $S \cap V_p$, so that $J_{\cV}+ J_{\cH} = \frj$ becomes standard on $U_p \subset \CC^2$. 
Thus, the computation above proves that $\o_a$ is $\frj$-tamed and K{\"a}hler on the
$\cup_{p\in \S^1}(B_{2\e_p}(p)-B_{\e_p}(p))-\S^*$. 
\end{proof}

\begin{remark} 
For a fixed surface $S$, the formula defining $\o_a$ near $S$ clearly extends to a non-degenerate closed $2$-form on $D_{\e_0}(\bar S) - \S^1$.
However, for different surfaces $S, S'$ these extended $2$-forms may differ in $D_{\e_0}(\bar S_i) \cap D_{\e_0}(\bar S_j) \subset B(\S^1)$. That is why we restrict the definition of $\o_a$ to $X - B(\S^1)$.
\end{remark}

To construct $\o_a^*$ we interpolate $\o_a$ with $0$ near $\S^1$; for that purpose we first prove that $\o_a$ admits a K{\"a}hler potential on a neighbourhood of $\S^1$. This neighbourhood consists of the union of annulus $A_p=B_{2 \e_p}(p) - B_{\e_p}(p) \subset X$ for each $p\in \S^1$; these are covered by the orbifold chart $U_{A_p}=B_{2 \e_p}(0) - B_{\e_p}(0)$. 

\begin{proposition} \label{prop:kahler-potential}
Let $p\in \S^1$; there is a K{\"a}hler potential 
$F_a \colon A_p \to [0,\infty)$ for the lifting of $\o_{a}$ to the chart $U_{A_p}$. 
That is, in $U_{A_p}$ we have
$$
\o_{a}= \frac{\ii}{2} \bd \bar \bd F_a .
$$
In addition, $a$ can be chosen so that there exists $0<t_0<t_1$ such that:
$$
B_{\e_p}(0) \subset F_a^{-1}([0, t_0)) \subset B_{3\e_p/2}(0)  \subset F_a^{-1}([0, t_1)) \subset B_{2\e_p}(0)  \, .
$$
\end{proposition}
\begin{proof}
First of all recall that the preparation of Lemma \ref{lem:omega^0} does not alter $\o|_{V_p}$, being $V=V_p$ a neighborhood of $p$ containing $B_{2\e_p}(p)$. To ease notation
let us suppose from now on that $V=B_{2\e_p}(p)$,
so $V$ is covered by an orbifold chart $U=B_{2 \e_p}(0) \to V$ with coordinates $(z,w)$.
Fix a surface $S \subset \S^*$. We cover $V-B_{\e_p}(p)$ with the charts $W= U_p-(B_{\e_p}(0) \cup_S N_\d(S))$ and $W'=\cup_S W'_S$, with $W'_S= N_{2\d}(S) \cap V$, so $V=W \cup W'$. The K{\"a}hler potential over $W$ is of course:
$$
F_a|_W= |z|^2 + |w|^2.
$$
We now look for the K{\"a}hler potential near a singular surface $S$.
Consider a rotation of $V$ in which $S$ corresponds to $w=0$. By remark \ref{rem:rotation}, on the set $W'_S=N_{2\d}(S) \cap V$ the expression of $\o_a$ is:
$$
\o_{a} = \frac{\ii}{2} \left(dz\wedge d\bar{z} + dw\wedge d\bar{w} \right)  - \frac{1}{4} dJ_Vd H(|w|^2,a)
$$
where $H(x)=\rho_{\d_2}(x) h(x,a)$. In addition, $ dJ_{\cV} d H(|w|^2,a) = d\frj d H(|w|^2,a)$ because $J_{\cV} +J_S = \frj$, and $dH(|w|^2,a)\in \cV^*$. Moreover, taking into account that
$\frj(d\zeta)=\ii d\zeta$ and $\frj(d\bar{\zeta})=-\ii d\bar{\zeta}$ for a complex variable $\zeta$,
we get $\frj \bd = \ii \bd$ and $\frj \bar{\bd}= - \ii \bar{\bd}$. Hence we obtain:
$$
d \frj d  = (\bd + \bar{\bd })\frj (\bd + \bar{\bd })= -2 \ii \bd \bar{\bd} \, .
$$
Thus, $- \frac{1}{4} dJ_Vd H(|w|^2,a) = \frac{\ii}{2} \bd \bar{\bd} H(|w|^2,a)$ and the K{\"a}hler potential is:
$$
F_a|_{W_S'} = |z|^2 + |w|^2 + H(|w|^2,a).
$$ 
Note that if $|w|>2\d_2$ then $H(|w|^2,a)=0$; thus $F_a|_{|w|>2\d_2}= |z|^2 + |w|^2$. 

If we consider another singular surface $S'$ with $p\in \bar{S}'$, we make another rotation in $V$ and repeat the process to construct $F_a$ near $S'$. Since transition functions are rotations, these functions glue together and give a function $F_a$ well-defined in $A$.
Note that, as discussed above, the global expression of $F_a$ in $A$ depends on both the radius $r^2=|z|^2 + |w|^2$
and the distance $d_S$ from a surface $S$. That is, we have in global coordinates $(z,w) \in A$ the expression:
$$
F_a(z,w)= |z|^2 + |w|^2 + \sum_{S}H(d_S(z,w)^2,a)
$$
where each $H(d_S(z,w)^2)$ extends as $0$ outside $N_{2\d_2}(S)$. Finally, the choice of $0<t_0<t_1$ with
$$
B_{\e_p}(0) \subset F^{-1}([0, t_0)) \subset B_{3\e_p/2}(0) \subset F^{-1}([0, t_1)) \subset B_{2\e_p}(0) 
$$
can be made for $a$ small enough. Indeed, the function $|z|^2 + |w|^2$ verifies the above property for $t_0= \frac{5}{4}\e_p$ and $t_1= \frac{7}{4}\e_p$, and $H_a(x)=H(x,a)$ are positive functions that converge uniformly to $0$ as $a \to 0$.
\end{proof}

We now prove a technical result that enables us to perform the desired interpolation.

\begin{lemma} \label{lem:extension}
Let $V \subset \CC^n$ open, and $F \colon V \to \RR$ a smooth function
such that $\tfrac{\ii}{2} \bd \bar \bd F$ is $\frj$-semipositive.
Let $h \colon \RR \to \RR$ smooth with $h'\ge 0$, $h''\ge 0$.
Denote $\o=\tfrac{\ii}{2} \bd \bar \bd F$, $\o_h=\tfrac{\ii}{2} \bd \bar \bd (h \circ F)$.

Then the form $\o_h$ is $\frj$-semipositive.
Moreover, $\o_h$ is $\frj$-positive on the subset of $V$ where
$\o=\tfrac{\ii}{2} \bd \bar \bd F$ is $\frj$-positive and $h'(F)>0$.
\end{lemma}
\begin{proof}
A computation in the complexified tangent bundle $TV \otimes \CC$ gives that 
$$
\tfrac{\ii}{2}\bd \bar \bd (h \circ F)=\tfrac{i}{2} h''(F) \bd F \wedge \bar \bd F + \tfrac{i}{2} h'(F) \bd \bar \bd F.
$$
On the other hand denote
$$
\b=\bd F \wedge \bar \bd F= \sum_{i,j}{(\bd_{z_i} F) (\bd_{\bar z_j}F) dz_i \wedge d \bar z_j} \, .
$$
Recall that $\b(v,\frj v)=-\tfrac{\ii}{2} \b(v-\ii \frj v, v+\ii \frj v)$
for every $v \in TV$, with $v-\ii \frj v \in T^{1,0}V$.
Take a vector $u=v-\ii \frj v=\sum_i a_i \bd_{z_i} \in T^{1,0}V$ and compute:
\begin{align*}
\b(u,\bar u)& = \sum_{i,j}{(\bd F \wedge \bar \bd F)( a_i \bd_{z_i},  \bar a_j \bd_{\bar z_j})}
= \sum_{i,j}{ a_i \bar a_j (\bd_{z_i} F) (\bd_{\bar z_j}F)}  \\
& = \sum_{i,j}{( a_i \bd_{z_i}F ) ( \bar a_j \bd_{\bar z_j}F )} =| \sum_{i}{ a_i \bd_{z_i}F} |^2
=|\bd F(u)|^2 \, .
\end{align*}
Here he have taken into account that $\bd_{\bar{z}_j} F= \overline{\bd_{z_j} F}$ because $F$ is real. 
This shows that $\tfrac{\ii}{2} \b(v, \frj v)=\tfrac{1}{4} \b(u,\bar u)= \tfrac{1}{4} |\bd F(u)|^2$ for $v \in TV$.
Finally, since $\o_h=\tfrac{\ii}{2} h''(F)\b + h'(F) \o$, the result is clear.
\end{proof}

Consider the K{\"a}hler potential for $\o_a$ in the chart $U_A$, given by
$F_a \colon U_A \to [0,\infty)$. As shown in Proposition \ref{prop:kahler-potential}, 
we can take numbers $t_1 >t_0>0$ so that 
$$
B_{\e_p}(0) \subset F_a^{-1}([0, t_0)) \subset B_{3\e_p/2}(0)  \subset F_a^{-1}([0, t_1)) \subset B_{2\e_p}(0)  \, .
$$
Let $h \colon \RR \to \RR$ be a function which vanishes for $t \le  t_0$, such that
$h(t)=t+c$ for $t \ge  t_1$, and with $h', h'' \ge 0$.
For instance one can take a bump function $\varrho$ with $\varrho' \ge 0$ so that $\varrho$ vanishes in $(-\infty,t_0)$ and equals $1$ in $(t_1,+\infty)$, and then define $h(t)=\int_{-\infty}^t \varrho$.

Let us define $\o_a^*=\frac{\ii}{2} \bd \bar \bd (h \circ F_a)$. This gives
a closed $2$-form in $U_A=B_{2\e_p}(0)-B_{\e_p}(0)$ which is $\frj$-semipositive by Lemma \ref{lem:extension} above; moreover it extends to $B_{\e_p}(0)$ as zero. The global formula on $U_A$ for the Kahler potential $F_a$ shows
that $F_a$ is invariant by the isotropy group $\G_p$, therefore $h \circ F_a$ is also $\G_p$-invariant.
On the other hand, as $\G_p$ acts by holomorphic maps, we have that
$\bar \bd \g^*=\g^* \bar \bd$ and $\bd \g^*=\g^* \bd$ as operators acting on forms, for any $\g \in \G_p$.

It follows that $\o^*_a= \bd \bar \bd (h \circ F_a)$ is $\G_p$-invariant in $U_p$. 
Since $\o_a^*$ equals $\o_a$ outside $B_{2 \e_p}(0)$, 
we see that $\o^*_a$ is a global orbifold $2$-form defined in $X$.
We summarize the above discussion in the following:
\begin{corollary} \label{cor:omega^*_a}
There exists a closed orbifold $2$-form $\o_a^*$ in $X$ satisfying:
\begin{itemize}
\item It vanishes in $B_{\e_p}(p)$.
\item It is $\frj$-positive in $B_{2 \e_p}(p)-(B_{\e_p}(p)\cup \S^*)$. In fact, $\o_a^*=\bd \bar \bd (h \circ F_a)$ there.
\item It coincides with $\o_a$ outside $B_{2 \e_p}(p)$.
\end{itemize}
\end{corollary}

\subsection*{Desingularisation}

As explained before, we now define a smooth atlas $\widehat \cA$ on $X-\S^1$ that makes the map $\Id \colon (X-\S^1,\cA) \to (X-\S^1,\widehat \cA)$ differentiable;  we also prove that $\widehat \o_a^*= \rId_*(\o_a^*)$ is the desired $2$-form. In order to make the presentation clearer, we first check on Proposition \ref{prop:manifold structure} that $\widehat \o_a = Id_*(\o_a)$ endows $(X-\cup_{p\in \S^1} B_{\e_p}(p), \widehat \cA)$ with the structure of a symplectic manifold. 

\begin{proposition} \label{prop:manifold structure} Notations and hipothesis as above.
 The following holds:
\begin{enumerate}
\item There is a manifold atlas $\widehat \cA= \{(\widehat U_\a,\widehat V_\a, \widehat \p_\a, \widehat \G_\a) \}$ on $X-B(\S^1)$
(i.e. an orbifold atlas with isotropy $\widehat \G_\a=\{1\}$) such that the identity
$$ 
\Id \colon (X-\S^1 ,\cA) \to (X-\S^1 ,\widehat \cA),
$$
is a smooth orbifold map, and it is a diffeomorphism out of $\S^*$.

\item The push-forward $\widehat \o_a=(\Id)_*(\o_a)$ is smooth on 
$(X- B(\S^1),\widehat \cA)$, and is a symplectic form for $a<a_2$.
In addition, on $(\cup_{p\in \S^1} (B_{2\e_p}(p)-B_{\e_p}(p)),\widehat \cA)$ we have that $\widehat \o_a$ is tamed by $\frj$.
\end{enumerate}
\end{proposition}
\begin{proof}
We shall modify some orbifold charts of $\cA$ to obtain $\widehat \cA$. 
First, if $x \notin \S^*$ we consider an orbifold chart $(U_x,V_x,\p_x,\{1\})\in \cA$ around $x$ with $V_x \cap \S^*=\emptyset$ and we take this as a chart of $x$ in $\widehat \cA$. 
Now, given a singular surface $S$ with isotropy isomorphic to $\ZZ_m$, we consider the cover of $D_{\e_0}(\bar S)$ as in Lemma \ref{lem:atlas bundle}. Take $(U_\a,V_\a,\G_\a,\p_\a)$ in this cover with $U_\a=S_\a \x D_{\e_0}$ and $p \notin V_\a$. We define $\widehat U_\a = S_\a \times  D_{(\e_0)^{m}}$, $\widehat V_\a=V_\a$, and $\widehat \p_\a(z',w')=\p_\a(z',w'^{\frac{1}{m}})$. 
Despite of the fact that $w'^{\frac{1}{m}}$ is not well-defined on $\CC$, the composition $\p_\a \circ (z',w'^{\frac{1}{m}})$ is because $\phi_\a$ is a $\G_\a$-invariant map.
The manifold coordinates $(z',w')$ of $\widehat \cA$ and the orbifold coordinates $(z,w)$ of $\cA$ are related by $w'=w^m, z'=z$.
We now check that the change of charts of $\widehat \cA$ are smooth. 
Denote 
$$
\psi_{\a \b}=(\psi_{\a \b}^1,\psi_{\a \b}^2) \colon U_\a \to U_\b
$$ 
the change of charts of the atlas $\cA$.
Let $V_\a \subset D_{\e_0}(\bar S)$ a chart from $\cA$ not containing any $p\in \S^1$, and take $V_\b$ another chart in $\cA$. Two cases arise:
\begin{enumerate}
\item If $V_\b \subset D_{\e_0}(\bar S)-\S^1$ we have
induced transition functions given by
$$
\widehat \psi_{\a\b} \colon \widehat U_\a \to \widehat U_\b \, , \quad
(z',w') \mapsto (\q_{\a\b}(z'),A_{\a\b}(z')^{m}w') =\psi_{\a \b}(z',w'^{\frac{1}{m}})
$$  
because
$$
\widehat \p_{\b} (\widehat \psi_{\a \b}(z',w'))= \widehat \p_\b( \psi_{\a \b}(z',w'^{\frac{1}{m}}))=\p_\a(z',w'^{\frac{1}{m}})=\widehat \p_\a(z',w').
$$ 
The map $\widehat \psi_{\a\b}$ is a diffeomorphism because $A_{\a \b}(z) \in \U(1)$.

\medskip

\item If $p \notin V_\b \not \subset D_{\e_0}(\bar S)$ and $V_\b \cap D_{\e_0}(\bar S) \ne \emptyset$, then by construction $V_\b \cap \S^*=\emptyset$. The induced change of chart in the atlas $\widehat \cA$ is 
$$
\widehat \psi_{\a \b}(z',w')=(\psi_{\a \b}^1(z',w'),\psi_{\a \b}^{2}(z',w')^m) \, ;
$$
this is a local diffeomorphism since $V_\b \cap \S^*=\emptyset$ and therefore $\psi^2_{\a\b}(z,w)\neq 0$.
\end{enumerate}

The map $\Id$ restricted to $D_{\e_0}(\bar S)$ is covered by the local maps: 
$$
\Id_\a \colon U_\a \to \widehat U_\a, \quad (z,w)\to (z,w^m)=(z',w'),
$$
which are diffeomorphisms outside $w=0$. Note that the radial function $r'=|w'|$ is again well-defined on $(D_{\e^m_0}(\bar S),\widehat \cA)$ and $\rId^*(r')=r^m=|w|^m$.

We now consider the symplectic form around a singular surface $S$; we follow the notation of Proposition \ref{prop:omega-a}.
First observe that if $\eta_\a= \pi^*(\n_\a) + d\h$, then  $(\rId)_*(\eta_\a)=\widehat \pi^*(\n_\a) + m \, d\h$, where $\widehat \pi \colon D_{\e^m_0}(\bar S) \to \bar S$ is the projection.
If we define $\eta'= \frac{1}{m}\Id_*(\eta)$, then $\eta'$ is a connection form in $D_{\e^m_0}(\bar S)$. Again, one can define smooth distributions $H'=\ker(r' dr' \wedge\eta')$ and 
$V'=\ker d \widehat \pi = \Id_*(\ker d \widehat \pi)$, and almost K{\"a}hler structures as before: 
$(\widehat \pi^*\o_S, J_S')$, $(r' dr' \wedge \eta', J_V'=\ii)$. 
Taking into account that $\Id^*(r')=r^m$ and the fact that $(\Id_*)J_V = J_V'$ (since $\Id$ is holomorphic), we obtain:
$$
(\rId)^*( \pi^*\o_S -\frac{1}{4} dJ_V'd(r'\,^2 + a^2)^{\frac{1}{m}} )= \o_a, \quad  r' \leq \d_2^{m}.
$$
Therefore $\widehat \o_a=\Id_*(\o_a)$ extends smoothly to $(X-B_{\e_p}(p),\widehat \cA)$, it is closed, 
and it is non-degenerate out of $S$. Moreover, near $S$ it has the form:
$$
\widehat \o_a=\pi^*\o_S -\frac{1}{4} dJ_V'd(r'\,^2 + a^2)^{\frac{1}{m}} \, .
$$
In every point of $S=\{r'=0\}$, taking into account the formula obtained for $dJ_Vd\bar{f}(r^2)$ in Proposition \ref{prop:omega-a}, the form $\widehat \o_a$ coincides with:
$$
\widehat \pi^*\o_S  +   \frac{1}{m a^{1-\frac{1}{m}}} r'dr'\wedge \eta', 
$$
which is $J_V' + J_H'$-tamed.

Take the set $V^*_S=(B_{2\e_p}(p)-B_{\e_p}(p)) \cap N_{\d_2}(S)$, covered by the chart  $U^*_S=  (B_{2 \e_p}(0) -  B_{ \e_p}(0)) \cap (\CC \x D_{\d_2}) \in \cA$; this has isotropy $\tilde{\G}=\{ \g \in \G_p \mbox{ s.t. }\g(z,0)=(z',0)\}$. Consider the induced chart $\widehat U^*_S \in \widehat \cA$ 
which has coordinates $(z',w')$ given by
$$
\Id_{U^*_S} \colon U^*_S \to \widehat U^*_S
\quad , \quad (z,w) \mapsto (z,w^m)=(z',w') \, .
$$
Its isotropy is $\tilde{\G}/\ZZ_m$, that acts without fixed points on $U^*_S$.
We claim that $\widehat \o_a$ is tamed by the standard complex structure $\frj$ on $\widehat U^*_S$.
Moreover, we can push-forward an almost complex structure on $(X-\S^1,\cA)$ to $(X-\S^1,\widehat\cA)$,
and near $\S^1$ this push-forward gives the standard almost complex structure.
Note that $\o_a$ is $\frj$-tamed in $U^*_S$ and, outside $S$, $\widehat \o_a$ 
coincides with $\o_a$ via the local biholomorphism $\Id_{U^*_S}$. 
Indeed, we saw that the form $\o_a$ was tamed on $U^*_S$ by $J_H + J_V= \frj$; outside $S$ we have
$$
J_{V}'+ J_{H}' = (\Id_{U^*_S})_*(J_V + J_H)= (\Id_{U^*_S})_*(\frj)=\frj
$$
the last equality since $\Id_{U^*_S}$ is holomorphic. Hence $\widehat \o_a$ is $\frj$-tamed in $\widehat U^*_S - S$, 
and also in $S \cap \widehat U^*_S$ because it is $J_V' + J_H'$-tamed on $S$ and that $J_V' + J_H'=\frj$ near $p$.
\end{proof}

To finish this section we extend the form $\widehat{\o}_a$ by zero as we did with $\o_a$ in Corollary \ref{cor:omega^*_a}:

\begin{corollary} \label{cor:widehat-omega_a^*}
There exists a closed orbifold $2$-form $\widehat{\o}_a^*$ in $(X-\S^1,\widehat{\cA})$ satisfying: 
\begin{itemize}
\item It vanishes in $B(\S^1)$.
\item It is $\frj$-positive in $\cup_{p\in \S^1} B_{2 \e_p}(p)-B_{\e_p}(p)$.
\item It coincides with $\widehat{\o}_a$ outside $\cup_{p\in \S^1}\bar{B}_{2 \e_p}(p)$. In particular it is symplectic there.
\end{itemize}
\end{corollary}
\begin{proof} 
Consider the orbifold symplectic form $\o_a^*$ on $(X,\cA)$ of Corollary \ref{cor:omega^*_a}.
We need to check that the form ${\Id}_*(\o_a^*|_{X-\S^*})$ extends 
to a closed two-form $\widehat{\o}_a^*$ on $(X-B(\S^1),\widehat \cA)$, 
and this form will have the required properties. 
As $\o_a^*=\o_a$ outside $B_{2 \e_p}(p)$ and $\Id_*(\o_a)=\widehat \o_a$, 
we only need to check that the push-forward of $\o_a^*$ extends on $B_{2\e_p}(p)$.

Let us consider an isotropy surface $S \subset \S^*$ and denote by $\widehat \p \colon \widehat U^*_S \to V^*_S$ 
the manifold chart in $\widehat \cA$ that we constructed in the proof of Proposition \ref{prop:manifold structure} 
in order to desingularize $V^*_S=(V-B_{\e_p}(p)) \cap N_{\d}(S)$.  The restriction of the identity map 
$$
\Id \colon (V^*_S, \cA) \to (V^*_S, \widehat \cA)
$$ 
is holomorphic, and its inverse is holomorphic on $V^*_S- S$.
This leads to the following equality on $V_S^*-S$:
$$
{\Id}_*(\o_a^*|_{X- \S^*})= \frac{\ii}{2} \bd \bar \bd (h \circ \widehat{F_a}), 
$$
where $h \colon \RR \to \RR$ is the smooth function constructed in Corollary \ref{cor:omega^*_a} and 
$$
\widehat F_a(z,w)= |z|^2 + |w|^{\frac{2}{m}} + 
\rho_{\d_2}(|w|^{\frac{2}{m}})\left( (|w|^{2} + a)^{\frac{1}{m}}- |w|^{\frac{2}{m}} \right).
$$
The function $\widehat F_a$ has a smooth extension defined on $V^*_S$ because near $w=0$ 
the expression of $\widehat F_a$ is $\widehat F_a(z,w)= |z|^2 + (|w|^{\frac{2}{m}} + a)^{\frac{1}{m}}$. 
Thus, we can extend $\Id_*(\o_a^*)$ over $S$. 
\end{proof}

\subsection{Symplectic orbifold structure with only isolated singularities.}

Now first extend our manifold atlas $\widehat \cA$ of $X-\S^1$ to an orbifold atlas of $X$ 
with only isolated singularities, and then we extend the symplectic form; ending up with
$(X,\widehat \cA,\widehat \o)$ a symplectic orbifold
with only isolated singularities. 

\subsection*{Extension of the orbifold structure.}

Let $p \in \S^1$ and let $(U,V,\G,\frj,\o_0)$ be a K{\"a}hler orbifold chart of $(X,\cA)$ around 
$p$. We have
$\G^* \lhd \G < \U(2)$, with $\G^*$ the isotropy group 
of the surfaces $S\subset \S^*$ accumulating at $p$ and $\G'=\G/\G^*$ the quotient, which acts in $U/\G^*$.
The manifold $(V-\{p\},\widehat \cA)$ has a complex structure induced
from the orbifold chart $(U-\{0\},\frj) \in \cA$, as was shown in Proposition \ref{prop:manifold structure}.
On the other hand, $V \cong U/\G$ has the structure of a complex orbifold induced by $\cA$. 
The identity map $\Id \colon (V-\{p\},\cA) \to (V-\{p\},\widehat \cA)$ 
is holomorphic and a biholomorphism out of $\S^*$. 
In both cases, the complex structure is the restriction to $U-\{0\}$ 
of the standard complex structure $\frj$ on $\CC^2$.

We also have a covering map $(U-\{0\})/\G^* \to (U-\{0\})/\G$ because $\G'$ acts freely on $(U-\{0\})/\G^*$. This allows us to consider the complex manifold $(U-\{0\}/\G^*,\widehat{\cA})$ and the complex orbifold $(U-\{0\}/\G^*,\cA)$; the complex structure is again in both cases induced from $\CC^2$, and the identity map
$(U-\{0\}/\G^*,\cA) \to (U-\{0\}/\G^*,\widehat{\cA})$ is holomorphic and
biholomorphic outside $\S^*$. 
The next proposition shows that the orbifold $(U-\{0\}/\G^*,\cA)$ can be naturally 
seen as an open set of $\CC^2$, 
allowing us to extend the complex structure $(U-\{0\}/\G^*,\widehat \cA)$ at the point $0$.

\begin{proposition}
The complex manifold structure on $((U-\{0\})/\G^*,\widehat \cA)$
can be naturally extended to a complex manifold structure on $U/\G^*$ so that
the group $\G'=\G/\G^*$ acts by biholomorphisms 
in the complex manifold $(U/\G^*,\widehat \cA)$.

In addition, there is an open set $\widehat U \subset \CC^2$ containing $0$,
a group $\G''$ acting on $\widehat U$ by biholomorphisms,
and a biholomorphic map $G \colon (\widehat U,\frj) \to (U/\G^*,\widehat \cA)$
such that $G$ is $(\G'',\G')$-equivariant.
\end{proposition}
\begin{proof}
As explained in the proof of Lemma \ref{lemma:trivial} there is an homeomorphism,
\begin{equation} \label{eq:H}
H\colon \CC^2 \to \CC^2 \quad , \quad H(z)=(f(z_1,z_2),g(z_1,z_2)).
\end{equation}
where $\{f,g\}$ is a basis of the algebra $\CC[z_1,z_2]^{\G^*}$ of $\G^*$-invariant polynomials.
This map induces an homeomorphism $\bar{H} \colon \CC^2/\G^* \to \CC^2$ which is holomorphic as an orbifold map and a biholomorphism out of the singular locus $\S^*$; here we have considered $\CC^2/ \G^*$ as a complex orbifold, covered by a unique chart $(\CC^2,\G^*)$. The structure that $(U-\{0\})/\G^*$ inherits when viewed as an open subset of $\CC^2/\G^*$ is precisely the orbifold structure determined by $\cA$. Let us call $G'=\bar{H}^{-1}$, define $\widehat U=H(U)\subset \CC^2$, so $U/\G^* \cong \widehat U$ via $\bar H$.
Let $G=\rId \circ G' \colon \widehat U \to U/\G^*$ and consider the restriction
$$
G| \colon (\widehat U-\{0\} , \frj)
\xrightarrow{G'} (U-\{0\}/\G^*, \cA) \
\xrightarrow{\Id} ((U-\{0\})/\G^*, \widehat \cA),
$$
which is bijective and biholomorphic out of $G|^{-1}(\S^*)$, 
and can be extended as a homeomorphism from $\widehat{U}$ to $U/\G^*$. 
The inverse $G|^{-1}$ is holomorphic out of $\S^*$,
being $\S^* \cap U$ a union of complex hyperplanes. 
Also, $G|^{-1}$ is a homeomorphism onto the set $\widehat U - \{0\} \subset \CC^2$ which is bounded.
By Riemann extension theorem, $G|^{-1}$ is holomorphic.
The inverse function theorem ensures that $G|$ is a biholomorphism.
This shows that the complex manifold structure in 
$((U-\{0\})/\G^*, \widehat \cA)$
can be extended naturally to all $U/\G^*$, in such a way that $G \colon (\widehat U , \frj) \to (U/\G^*, \widehat \cA)$ is a global complex chart, hence a biholomorphism.

We consider $\G '' = \{ \g''=G^{-1} \circ [\g] \circ G \colon [\g] \in \G'=\G/\G^* \}$.
Note that since $\g\in U(2)$ then clearly the action of $\G'$ in $(U/\G^*, \widehat \cA)$ 
is holomorphic outside the isotropy, i.e. on $((U-\S)/\G^*,\widehat \cA)$ with $\S=\S^* \cup \{0\}$ a complex subvariety. Again by Riemann extension theorem, the action of $\G'$ must be holomorphic in all $(U/\G^*,\widehat \cA)$. 
Since $G$ is a biholomorphism, every $\g'' \in \G''$ is a biholomorphism of $(\widehat U,\frj)$. Hence $\G''$ acts in $\widehat U$ by biholomorphisms. 

We call $q \colon U/\G^* \to (U/\G^*)/\G' \cong U/\G$ the quotient map.
Consider $Y=(U/\G^*,\widehat \cA)$ a complex manifold and $\G' \cong \G''$ equivalent groups acting on $Y$ by biholomorphisms, so the space $Y/\G'$ is a complex orbifold.
In addition, $(\widehat U,Y/\G', \p_0, \G'')$ gives a global orbifold chart of $Y/\G'$, 
with $\p_0= q \circ G \colon \widehat U \to Y/\G'$ the orbifold chart that induces $\bar{\p_0} \colon \widehat  U/\G'' \to Y/\G'$ a homeomorphism.

Finally, using the homeomorphism $h \colon Y/\G' \to V$ given by 
$Y/\G' = (U/\G^*)/\G' \cong U/\G \cong V$ 
we have $(\widehat U,V, \widehat\p, \G'')$ with $\widehat\p=h \circ \p_0$ ; this gives an orbifold chart around the point $p \in X$ which is compatible with
the manifold structure $(X - \S^1, \widehat \cA)$.
\end{proof}

\begin{corollary} \label{cor:new-orb-str}
The map $G$ induces an orbifold chart $(\widehat U,V,\widehat \p, \G'')$ of $V=V^p$ which is compatible
with the manifold structure $(X-\S^1,\widehat \cA)$.
\end{corollary}

\subsection*{Symplectic form on $(X,\widehat \cA)$}

Adding to $\widehat \cA$ the charts defined on Corollary \ref{cor:new-orb-str}
we obtain an orbifold atlas $\cA'$ on $X$ with isolated singularities.
We also have a symplectic form $\widehat \o_a^*$ on $(X-B(\S^1),\widehat \cA)$ 
given by Proposition \ref{prop:manifold structure}. 
The last step now is extending the symplectic form $\widehat \o_a^*$ to all the orbifold $(X, \cA')$.
For this it will be useful the following Lemma.

\begin{lemma} \label{lem:gluing_forms}
Denote $B_r=B_r(0)$ a ball of radius $r$, and let $U \subset \CC^n$ open set containing $B_r$.
Let $\o_1, \o_2 \in \O^2(U)$ be closed $2$-forms so that:
\begin{itemize}
\item The form $\o_1$ vanishes on $\bar B_{\e_1}$, it is $\frj$-semipositive
in $B_{\e_2} - \bar B_{\e_1}$, and it is $\frj$-positive in $U - B_{\e_2}$, for some $\e_2 < \e_1 < r$.
\item The form $\o_2$ is non-degenerate in $U$ and $\frj$-tame.
\end{itemize}
Then, for any choice of $\e_3$ with $r> \e_3 > \e_2$ there is a $\frj$-tame symplectic form $\o$ in $U$ 
so that $\o|_{B_{\e_1}}=\d \o_2$ for some $\d>0$ small, $\o=\o_1$ outside $B_{\e_3}$.
\end{lemma}
\begin{proof}
Let $\rho=\rho_\e(r)$ be a radial bump function which equals $1$ in $0 \le r \le \e_2$
and equals $0$ in $r \ge \e_3$. Let $\b \in \O^1(B_r)$ such that $d \b=\o_2$.
Let us define $\o=\o_\d= \o_1 + \d d (\rho \b)$.
We have that $\o= \d \o_2$ on $r\le \e_1$, so it is symplectic and $\frj$-tame there.
On $\e_1 \le r \le \e_2$ we have $\o= \o_1 + \d \o_2$; as $\o_1$ is $\frj$-semipositive
and and $\o_2$ is $\frj$-positive in $B_{\e_2} - B_{\e_1}$, we see that $\o$ 
is $\frj$-positive in $B_{\e_2} - B_{\e_1}$. Also, $\o=\o_1$ on $r \ge {\e_3}$.

Finally, on $ \e_2 \le r \le \e_3$ we have $\o=\o_1+ \d d\rho \wedge \b + \d \rho \, \o_2$.
Since $\o_1$ is $\frj$-positive on the compact annulus $ \e_2 \le r \le \e_3$,
there exists a constact $C>0$ with $\o_1(u,\frj u) \ge C |u|^2$ for all $u \in \RR^{2n}$
and all point in the annulus.
Hence 
\begin{align*}
|\o(u,\frj u) |
&  =|\o_1(u,\frj u)+ \d d\rho \wedge \b(u,\frj u) + \d \rho \, \o_2(u,\frj u)| \\
& \ge \o_1(u,\frj u) + \d \rho \, \o_2(u,\frj u) - |\d d\rho \wedge \b(u,\frj u)| \\
& \ge C |u|^2 - \d \|d\rho \wedge \b \| |u|^2 \\
& =(C- \d \|d\rho \wedge \b \|)|u|^2 \, 
\end{align*}
so if $\d < \frac{C}{\|d\rho \wedge \b \| + 1}$ then $\o$ is $\frj$-tame and symplectic.
\end{proof}

Now recall the closed $2$-form $\widehat \o_a^*$
of Corollary \ref{cor:widehat-omega_a^*}. The form $\widehat \o_a^*$ is defined on $(X,\cA')$,
vanishes in $B(\S^1)$, coincides with $\widehat \o_a$ outside $B_{2 \e_p}(p)$,
and it is $\frj$-positive in $\cup_{p\in \S^1} (B_{2 \e_p}(p)-B_{\e_p}(p))$.
By Lemma \ref{lem:gluing_forms} we can glue $\widehat \o_a^*$ with the standard symplectic form $\o_0$ near $p$ 
to construct an orbifold symplectic form on $(X,\cA')$ extending $\widehat \o_a$.

\begin{corollary}
There exists an orbifold symplectic form $\bar \o_a$ on $(X,\cA')$ 
which coincides with $\widehat \o_a$ out of some neighborhood of $\S^1$.
\end{corollary}
\begin{proof}
Consider the orbifold chart $(\widehat U,V,\widehat \p, \G'')$ around a point $p\in \S^1$ of Corollary \ref{cor:new-orb-str}.
Consider a local representative of $\widehat \o_a^*$ in the chart $\widehat U$, denote it $\o_1=\widehat \o_a^*$.
Consider also $\o_2=-\frac{\ii}{2}( dz \wedge d \bar z + dw \wedge d \bar w)$ 
the standard symplectic form in $\widehat U \subset \CC^2$. Take balls $B_{\e_i}$ in $\widehat U$
so that 
$$
B_{\e_1} \subset \widehat \p^{-1}(B_{\e_p}(p))  \subset \widehat \p^{-1}(B_{2\e_p}(p)) \subset B_{\e_2} \subset B_{\e_3} \, .
$$
We have that $\o_1$ vanishes in $B_{\e_1}$, it is $\frj$-semipositive in $B_{\e_2}- B_{\e_1}$ and coincides with
$\widehat \o_a$ outside $B_{\e_2}$, so it is $\frj$-positive there. We are in the hypothesis of Lemma \ref{lem:gluing_forms},
and this gives our desired symplectic form $\bar \o_a$ in $\widehat U$ with $\bar \o_a=\widehat \o_a$
outside $B_{\e_3}$. The only point is that $\bar \o_a$ may not be
$\G''$-invariant; in case it is not, replace it by its average over $\G''$, which is also $\frj$-tame because diffeomorphisms on $\G''$ are holomorphic. Being $\o_1$ invariant under $\G''$, the average coincides with $\widehat \o_a$ outside $B_{\e_3}$.
\end{proof}
\begin{corollary}
The symplectic orbifold $(X, \cA', \bar \o_a)$ has only isolated singularities.
\end{corollary}

\subsection{Cohomology groups of the resolution}

The computation of the cohomology groups of the resolution can be obtained from the results in \cite{CFM}.

\begin{proposition}\label{prop:cohom}
Let $\pi \colon (\widetilde X, \widehat \o) \to (X,\o)$ be a symplectic resolution of a symplectic orbifold. 
Define the subset of $\S^1$
$$
\Delta = \{ x \in \S^1 \mbox{ s.t. } \G_x/\G^*_x \neq \{1\} \},
$$
where $\G^*_x$ is the subgroup of $\G_x$ generated by the isotropy surfaces accumulating at $x$. 
For each $p\in \Delta \cup \S^0$, let $E_p=\pi^{-1}(p)$ be the exceptional set. 
For $k>0$ there is a short exact sequence:
$$
0 \rightarrow H^k(X) \xrightarrow{\pi^*} H^k(\widetilde{X}) \xrightarrow{i^*} \bigoplus_{p \in \S^0\cup \Delta} H^k(E_p) \rightarrow 0.
$$
\end{proposition}
\begin{proof}
The symplectic resolution of $(X,\cA,\o)$ is divided into two steps; we first perform a partial resolution $(X,\widehat \cA, \widehat \o) \to (X,\cA,\o)$. The underlying topological space of the partial resolution does not change but its singularities are isolated and consists precisely of the points in $\Delta \cup \S^0$.
After this, we construct a resolution  $(\widetilde{X}, \widetilde{\cA}, \widetilde{\o})\to(X,\widehat \cA, \widehat \o)$ employing the method described in \cite[Theorem 3.3]{CFM}.
The cohomology ring of $\widetilde{X}$ was computed in \cite[Proposition 3.4]{CFM} and implies the statement.
\end{proof}

\section{Examples} \label{sec:examples}

In this section we give some examples of $4$-orbifolds to which the resolution described above can be applied.

\subsection*{Products of orbifolds}

Let $(S, \o)$ be a compact symplectic $2$-dimensional orbifold. 
Its isotropy set consists of an isolated set of points $\{p_0,\dots, p_n \}$; we denote the isotropy group of $p_j$ by $\mathrm{G}_j$.

Consider the product orbifold $(S \times S, \o + \o)$ and the symplectic involution $R(x,y)=(y,x)$ . 
Let us define the symplectic orbifold $X=S \times S / \ZZ_2 $, where $\ZZ_2=\{ R, \rId \}$, 
and denote $q \colon S \x S \to X$ the projection to the orbit space. The isotropy set of $X$ is $\S^* \cup \S^1$, where:
\begin{enumerate}
\item $\S^*= q (\cup_{j=1}^n (S-\{p_1,\dots,p_n\}) \times \{p_j\})  \cup 
q (\{(x,x), x \in S-\{p_1,\dots,p_n\}\})$,
\item $\S^1=q (\{(p_j,p_k), \quad 1 \leq j \leq k \leq n\})$.
\end{enumerate}
The isotropy group of points on $q ((S-\{p_1,\dots,p_n\}) \times \{p_j\})$ is $\mathrm{G}_{j}$, and for points on $q (\{(x,x), x \in S-\{p_1,\dots,p_n\}\})$ it is $\ZZ_2$. If $j<k$, the isotropy group of $(p_j,p_k)$ is $\mathrm{G}_{jk}=\mathrm{G}_j \times \mathrm{G}_k$. If $j=k$ a presentation of the isotropy group
$\mathrm{G}_{jj}$ is $ \la \xi , \mathrm{G}_j \times \mathrm{G}_j  |\quad  \xi^2= 1, \quad \xi (\g,\g')= (\g',\g) \xi \ra $. 
Indeed if $(U,V,\p,\G_j)$  is an orbifold chart around $p_j$ on $S$, 
then an orbifold chart around $q(p_j,p_j)$ on $X$ is:
$$
(U\x U, q (V\times V), q \circ (\p \times \p), G_{jj})
$$
where the action of $\xi$ is given by $\xi(z,w)=(w,z)$, and the action of $\mathrm{G}_j \x \mathrm{G}_j$ 
is $(\g,\g')(z,w)=(\g z, \g' w)$.

Theorem \ref{thm:main-thm} allows us to obtain a symplectic resolution of the orbifold $(X,\o)$; 
this resolution is homeomorphic to $X$ because one can check that $\mathrm{G}_{jk}'= \{1\}$, 
following the notation of Lemma \ref{lemma:trivial}.

\medskip

Now let $(S',\o')$ be a compact $2$-dimensional symplectic orbifold (possibly different from $(S,\o)$),
with singularities $\{ p_0',\dots, p_m'\}$ and isotropy groups $\G_1',\dots, \G_m'$. 
We can also consider the product orbifold $(S \x S', \o + \o')$.
The isotropy set is $\S^*\cup \S^1$, with $\S^1=\{ (p_j,p_k'), \quad 1 \leq j \leq n, 1 \leq k \leq m\}$. The isotropy group of $(p_j,p_k)$ is $G_{jk}=\G_j \x \G_k'$ and verify that $G_{jk}'=\{1\}$. 
The orbifold $(S \times S',\o + \o')$ verifies the hypothesis of Theorem \ref{thm:main-thm}, and its resolution is homeomorphic to $S \x S'$. Note that one could also have constructed the resolution as $(\widetilde S \x \widetilde S', \widetilde \o + \widetilde \o')$, where $q \colon (\widetilde S, \widetilde \o) \to (S,\o)$ 
and $q \colon (\widetilde S, \widetilde \o) \to (S,\o)$ 
are the symplectic resolutions provided in \cite{CFM}.

\subsection*{Mapping torus over a surface of genus $2$}

Consider $\S_2$ a genus $2$ surface smoothly embedded in $\RR^3$ with coordinates $(x,y,z) \in \RR^3$.
We require that $\S_2$ is symmetric with respect to the planes $\{x=0\}$, $\{y=0\}$ and $\{z=0\}$.
Consider the symplectic form in $\S_2$ given by $\o_{\S_2}= \iota_N (\vol_3)|_{\S_2}$, 
being $N$ the outer unit lenght normal to $\S_2$, and $\vol_3=dx \wedge dy \wedge dz$ the volume form of $\RR^3$.
Consider the maps $\p(x,y,z)=(-x,y,-z)$, $\g(x,y,z)=(-x,-y,z)$; these restrict to symplectomorphisms
of $(\S_2, \o_{\S_2})$ since they preserve $N$ and $\vol_3$.

\begin{center}
\includegraphics[scale=0.6]{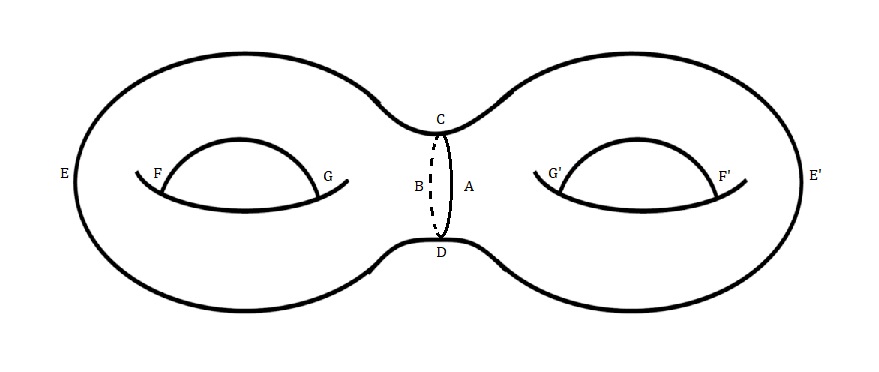}
\end{center}

Consider $M_\g(\S_2)$ the mapping torus of $\S_2$ by $\g$; that is, 
 $M_\g(\S_2)= (\S_2 \x I) / \hspace{-1.5mm} \sim$  where $(p,1) \sim (\g(p),-1)$ and $I=[-1,1]$.
In the space $M_\g(\S_2) \x S^1$ we lift the action of $\p$ as 
$$
\p([p,t],s)=([\p(p),t],s)
$$
for $[p,t] \in M_\g(\S_2)$, $s \in S^1=[-1,1]/ \hspace{-1mm} \sim$. 

Note that this action is well defined because if we take
$(p,1)$ and $(\g(p),-1)$ two representatives of the same class, they get mapped to $(\p(p),1)$
and $(\p(\g(p)),-1)=(\g(\p(p)),-1)$, so their images represent the same class also. 
Take also the map $\xi$ acting on $M_\g(\S_2) \x S^1$ as 
$$
\xi([p,t],s)=([p,-t],-s) \, .
$$
The above action is well-defined
because $(p,1,s)$ and $(\g(p),-1,s)$ are mapped to $(p,-1,-s)$ and $(\g(p),1,-s)$,
and $(\g(p),1) \sim (p,-1)$ since $\g^2=\Id$.
On the other hand let us consider the symplectic form on $M_\g(\S_2) \x S^1$ given in coordinates as 
$$
\o=\o_{\S_2} + dt \wedge ds.
$$
Near a point $([p,1],s) = ([\g(p),-1],s) \in M_\g(\S_2) \x S^1$ we can take a chart of the form 
$$
U^p \x (1-\e,1] \x (s-\e, s + \e) \cup \g(U^p) \x [-1,-1+\e) \x (s-\e, s + \e)
$$
where the above expression for $\o$ is well-defined, since $\g$ is a symplectomorphism of $\S_2$.

We can describe $M_\g(\S_2) \x S^1$ in an alternative manner. 
Consider $Y=\S_2 \x \CC^2$ and the isometries of $Y$, 
$\t_1(p,w)=(\g(p),w + 1)$, $\t_2(p,w)=(p,w + i)$. 
These determine a $\ZZ^2$-K\" ahler action on $Y$ and $M_\g(\S_2) \x S^1= Y / \ZZ^2$, 
hence $M_\g(\S_2) \x S^1$ is K\" ahler.

Note that in the symplectic manifold $(M_\g(\S_2) \x S^1,\o)$ the group $\G=\langle \p , \xi \rangle \cong \ZZ_2 \x \ZZ_2$ acts by symplectomorphisms.
We define a $4$-orbifold $X$ as 
$$
X= \frac{M_\g(\S_2) \x S^1}{\langle \p, \xi \rangle}
$$
so $(X,\o)$ is a symplectic orbifold.

Let us study the isotropy subset of $X$. 
We may abuse notation and identify
the isotropy points of $X$ with the isotropy points of the action
of $\langle \p , \xi \rangle$ in $M_\g(\S_2) \x \S^1$; the context should clarify each case.
The maps $\p, \g, \g \circ \p \colon \S_2 \to \S_2$ have the following fixed points 
$$
\fix(\p)=\{A, B\} \, , \,  \fix(\g)=\{C,D\} \, , \,  \fix(\g \circ \p)=\{E,F,G, E',F',G'\}  \subset \S_2
$$
with $A=(0,1,0), B=(0,-1,0), C=(0,0,1), D=(0,0,-1)$,
and $\fix(\g \circ \p)$ corresponds to the six points of intersection of $\S_2$ with the $x$-axis.
Note also that $\g(A)=B, \g(B)=A$, $\p(C)=D$, $\p(D)=C$, 
and $\p(E')=E$, $\p(F')=F$, $\p(G')=G$.

The isotropy points for the group $\langle \p, \xi \rangle$ acting on $M_\g(\S_2) \x S^1$ are as follows:
\begin{itemize}
\item Isotropy surfaces given by
\begin{align*}
S_\p & = \{([A,t],s) \mbox{ s.t. } (t,s) \in I^2\} \cup \{([B,t],s) \mbox{ s.t. } (t,s) \in I^2\} \, \, , \\
S^1_\xi & =\{([p,0],0) \mbox{ s.t. } p \in \S_2\} \, , \\
S^2_\xi & =\{([p,0],1) \mbox{ s.t. } p \in \S_2\} \, .
\end{align*}
Note that $S_\p$ is a torus, since $(A,1,s) \sim (B,-1,s)$, and $S^i_\xi$ are
surfaces with genus $2$, identified with $\S_2$.
The generic points of $S_\p$ have isotropy $\langle \p \rangle \cong \ZZ_2$, 
and those of $S_\xi$ have isotropy $\langle \xi \rangle \cong \ZZ_2$.

\medskip

\item The intersection of $S_\p$ and the $S^i_\xi$ are the points $A_0=([A,0],0)$,
$B_0=([B,0],0)$, $A_1=([A,0],1)$, and $B_1=([B,0],1)$; these are points of isotropy $\langle \p, \xi \rangle \cong \ZZ_2 \x \ZZ_2$.

\medskip

\item Eight isolated isotropy points. Two of them, $C_1=([C,1],1)$ and $D_1=([D,1],1)$,
have isotropy $\langle \xi \rangle \cong \ZZ_2$; the rest of them are the points $E_1=([E,1],1)$, $F_1=([F,1],1)$, 
$G_1=([G,1],1)$, $E'_1=([E',1],1)$, $F'_1=([F',1],1)$, 
$G'_1=([G',1],1)$, all with isotropy $\langle \p \circ \xi \rangle \cong \ZZ_2$.
\end{itemize}

Of the above fixed points in $M_\g(\S_2) \x S^1$ not all of them are different in the quotient $X$: we have $E_1 \sim E'_1$, $F_1 \sim F'_1$, $G_1 \sim G'_1$.
Moreover $S^i_\xi$ becomes a torus $\S_2/\langle \p \rangle$ in $X$,
and $S_\p$ becomes a sphere. 

Following the previous notation for the isotropy points of an orbifold $X$,
the isotropy subset $\S$ of $X$ decomposes as $\S=\S^* \cup \S^1 \cup \S^0$,
with $\S^1=\{A_0,B_0,A_1,B_1\}$, $\S^*= (S_\p \cup S^1_\xi \cup S^2_\xi) - \S^1$,
and $\S^0=\{C_1,D_1,E_1, F_1, G_1\}$.

Now we compute the betti numbers of $X$. 
For this it is useful to express $X$ in an alternative way.
Recall that the quotient $T=\S_2/\hspace{-0.2mm}\langle \p \rangle$ is a torus; its fundamental domain
being $D_T=\S_2 \cap \{x \ge 0\} \subset \RR^3$ with identifications $(0,y,z) \sim (0,y,-z)$.
The map $\g \colon \S_2 \to \S_2$ commutes with $\p$, so it descends to a homeomorphism of $T$.
Consider the mapping torus 
$$
M_\g(T)= (T \x [-1,1])/\sim
$$ 
where $([p],1) \sim ([\g(p)],-1)$.
It is immediate to check that $X = ( M_\g(T) \x S^1 ) / \langle \xi \rangle$.

The following well-known Lemma is necessary for the computation of the fundamental group of $X$.
\begin{lemma} \label{lem:mapping-torus}
Let $T$ be a $CW$-complex, and $\g \colon T \to T$ a homeomorphism which fixes a point $x_0 \in T$.
Let $M_\g(T)= T \x [0,1]/\hspace{-1mm}\sim$ with $(x,0) \sim (\g (x),1)$.
Then $\pi_1(M_\g(T)) \cong \pi_1(S^1) \ltimes_{\g_*} \pi_1(T)$.
\end{lemma}
\begin{proof}
Recall first that the operation in $\pi_1(S^1) \ltimes_{\g^*} \pi_1(T)$
is 
$$
(n,g) \cdot (n',g')=(n+n', g \cdot  \g_*^n(g')) \, ,
$$
where $\g_* \colon \pi_1(T,x_0) \to \pi_1(T,x_0)$ is the induced map.

We have a bundle structure on $M_\g(T)$ given by $T \xrightarrow{i} M_\g(T) \xrightarrow{\pi_*} S^1$,
where $i(x)=[x,0]$ and $\pi([x,t])=t$.
This gives a short exact sequence 
$$
1 \to \pi_1(T) \xrightarrow{i_*} \pi_1(M_\g(T)) \xrightarrow{\pi_*} \pi_1(S^1) \to 1 \, .
$$
There is a section $s \colon S^1 \to M_\g(T), t \mapsto (x_0,t)$; it is well-defined because $\g(x_0)=x_0$.
This gives $s_* \colon \pi_1(S^1) \to \pi_1(M_\g(T))$
a right inverse for $\pi$, which gives an splitting of the above short exact sequence; 
then $\pi_1(M_\g(T))$ is the semi-direct product of $\pi_1(T)$
and $\pi_1(S^1)$, where the action of $\pi_1(S^1)$ in $\pi_1(T)$ is by conjugation.

Let us call $\a=s_*(1)$, where $1 \in \pi_1(S^1)$ is the generator.
Note that $\a(t)=[(x_0,t)]$, $t \in [0,1]$.
It only remains to see that every $g \in \pi_1(T)$ satisfies that 
$\a g  \a^{-1}=\g_*(g)$ in $\pi_1(M_\g(T))$.

Consider the homotopy $H \colon S^1 \x [0,1] \to M_\g(T)$ given as
$$
H_s(t)=
\begin{cases}
(x_0,3ts),     & t\in[0,\tfrac{1}{3}]\\
(\g ( g(3t-1)),s), & t\in[\tfrac{1}{3},\tfrac{2}{3}]\\
(x_0,3(1-t)s), & t\in[\tfrac{2}{3}, 1].
\end{cases}
$$
It is immmediate to check that $[H_0]= \g_*(g)$ and $[H_1] =  \a g \a^{-1}$, proving the Lemma.
\end{proof}

Now we compute the fundamental group of $M_\g(T)$, with $T=\S_2/\langle \p \rangle$ as above. 
Take as base point $[C]=[D] \in T$,
which is a fixed point by $\g$, and choose generators $a, b$ for $\pi_1(T)$ so that
a representative for $a=[\a]$ in the fundamental domain $D_T$ is the circle 
$$
\a=D_T \cap \{z=1\}=\{(2+\cos t, \sin t, 1) \colon 0 \le t \le 2 \pi\} \, .
$$
Similarly, a representative for $b=[\b]$ is a semicircle
$$
\b=\{(\cos t, 0, \sin t) \colon \pi/2 \le t \le 3 \pi/2\}
$$
going from $C$ to $D$ in $D_T \cap \{y=0\}$; $\b$ descends 
to a loop in the quotient $T=D_T/\sim$.
By Lemma \ref{lem:mapping-torus}, the fundamental group of $M_\g(T)$ is
$$
\pi_1(M_\g(T)) \cong \pi_1(S^1) \ltimes_{\g_*} \pi_1(T)  \cong \ZZ \ltimes_{\g_*} \ZZ^2
$$
with operation $(n,x) \cdot (n',x')=(n+n', x + (\g_*)^n(x'))$,
being $\g_* \colon \pi_1(T) \to \pi_1(T)$ the automorphism induced by $\g \colon T \to T$
in $\pi_1(T)=\pi_1(T,[C])$.

In order to compute $\g_*$, we take the representatives in $D_T$ of $a$ and $b$ described above and compute
their image by $\g_*$; note that $\g$ seen as a map in $\S_2$ does not map $D_T$ to itself,
but $\p \circ \g(x,y,z)=(x,-y,-z)$ does, 
and both maps induce the same map on the quotient $T=\S_2/\langle \p \rangle$. 
The loop $a=[(2+\cos t, \sin t, 1)]$, $0 \le t \le 2 \pi$,
is mapped to $\p \circ \g (a)=[(2+\cos t, -\sin t, -1)]$, and this is a circle in $D_T \cap \{z=-1\}$
homotopic to $a$ but with the opposite orientation as $a$, so $\g_*(a)=-a$.
Similarly, $b=[(\cos t, 0, \sin t)]$, $\pi/2 \le t \le 3 \pi/2$, 
is mapped to $\p \circ \g (b)=[(\cos t, 0, -\sin t)]$, 
again the same circle but with opposite orientation, so $\g_*(b)=-b$. 
We conclude that 
$$
\g_*=-\Id \colon \pi_1(T) \to \pi_1(T) , \, x \mapsto -x \, .
$$
It follows that $\pi_1(M_\g(T)) \cong \ZZ \ltimes \ZZ^2$
with operation given by 
$$
(n,x)\cdot(n',x')=(n+n',x + (-1)^n x') \, .
$$
We claim that the abelianization of this group is $H_1(M_\g(T),\ZZ) \cong \ZZ \x \ZZ_2 \x \ZZ_2$.
Indeed, if we impose the condition that $(1,x) \cdot (0,x)$ and $(0,x) \cdot (1,x)$ coincide
we get that $(1,0)$ equals $(1,2x)$, hence $2x=0$ for all $x$ in the abelianization. 
This applies to the generators $a, b$. Once we impose that in the abelianization every $x$ equals $-x$, 
the operation $\cdot$ becomes commutative, hence the claim.

From this it follows that
\begin{align*}  
\pi_1(M_\g(T) \x S^1) & \cong (\ZZ \ltimes \ZZ^2) \x \ZZ \\
H_1(M_\g(T) \x S^1,\ZZ) & \cong \ZZ \x \ZZ_2 \x \ZZ_2 \x \ZZ \cong \ZZ_2^2 \x \ZZ^2 \, .
\end{align*}
Note that torsion part of the homology comes from the torus $T$ and the free part comes from the two circles
associated to the coordinates $(s,t)$.
When passing to real coefficients we can consider de Rham cohomology
and we get $H^1(M_\g(T) \x S^1,\RR)= \langle dt, ds \rangle$.
As the action of $\xi$ in $M_\g(T) \x S^1$ sends $dt, ds$ to $-dt, -ds$,
it follows that the cohomology of the orbifold $X=(M_\g(T) \x S^1)/\langle \xi \rangle$ is 
the $\xi$-invariant part of $\langle dt, ds \rangle$, i.e. $H^1(X,\RR)=0$.

\medskip

Now let us compute the fundamental group of $X$. Recall that $M_\g(T) \x S^1$
is a torus bundle over a torus, i.e. $T \to M_\g(T) \x S^1 \to S^1 \x S^1$
where fibers are given by $T=\{([p,t_0],s_0) \mbox{ s.t. } p \in T\}$ and the bundle map
sends $([p,t],s)$ to $(t,s)$.
We have a short exact sequence
$$
1 \to \pi_1(T) \xrightarrow{i_*} \pi_1(M_\g(T) \x S^1) \xrightarrow{\pi_*} \pi_1(S^1 \x S^1) \to 1
$$
where $i \colon T \to M_\g(T) \x S^1,  \,  p \mapsto ([p,t],s)$ is the inclusion of the fiber $F_{(t,s)} \cong T$,
and the bundle map is $\pi \colon M_\g(T) \x S^1 \to S^1 \x S^1 , \, ([p,t],s) \mapsto (t,s)$.
Consider 
$$
q \colon M_\g(T) \x S^1 \to X= (M_\g(T) \x S^1)/\langle \xi \rangle
$$
the quotient map. Take as base points $A_0$ and $q(A_0)$ respectively.
Since $A_0$ is fixed by $\xi$, we have $q^{-1}(q(A_0))=\{ A_0 \}$.
This gives that $q_* \colon \pi_1(M_\g(T) \x S^1) \to \pi_1(X)$ is an epimorphism by \cite[Corollary 6.3]{Bre}.

In $M_\g(T) \x S^1$ there are two fibers invariant by the action of $\xi$ and not formed by fixed points, namely $F_{(1,0)}$ and $F_{(1,1)}$. Let us take as base points $A$, $([A,1],0)$ and $(1,0)$ respectively.
Call $F \cong T$ any of these fibers. Under the quotient map $q$, $F$ is mapped to $q(F) \cong T/\langle \g \rangle$. This is so because 
\begin{align*}
\xi([p,1],0)=([p,-1],0)=([\g(p),1],0) \, ,\\
\xi([p,1],1)=([p,-1],-1)=([\g(p),1],1) \, ;
\end{align*}
hence $q \circ i(p)=q \circ i(\g(p))$ for $p \in T$,
being $i \colon T \to F \subset M_\g(T) \x S^1$ the inclusion. 
Recall that $q(F)=F/\langle \g \rangle \cong \S_2 /\langle \p, \g \rangle \cong S^2$
is topologically a sphere, so we call $S^2=T/\langle \g \rangle$.
The map $q_* \circ i_* \colon  \pi_1(F) \to \pi_1(M_\g(T) \x S^1)$
factors through $\pi_1(S^2)=\{1\}$, so it is constant.
Hence $\im(i_*)=\ker (\pi_*) \subset \ker(q_*)$, so the map $q_*$
induces a map $\bar q_* \colon \pi_1(S^1 \x S^1) \to \pi_1(X)$
in the quotient $\pi_1(M_\g(T) \x S^1)/\pi_1(F) \cong \pi_1(S^1 \x S^1)$.

Note that $\pi_1(S^1 \x S^1)$ can be seen as a subgroup of $\pi_1(M_\g(T) \x S^1)$
via the section 
$$
f \colon S^1 \x S^1 \to M_\g(T) \x S^1 , \, (t,s) \mapsto 
\begin{cases}
([A,1+2t],s) \, ,     & t\in[-1,0]\\
([B,-1+2t],s) \, , & t\in[0,1].
\end{cases}  \, .
$$

The image of $f$ is precisely the isotropy surface $S_\p$,
whose image by $q$ is $q(S_\p)=S_\p/\langle \xi \rangle \cong S^2$,
homeomorphic to a sphere.
As $\bar q_*=q_* \circ f_*$ factors through $\pi_1(q(S_\p))={1}$,
we see that $\bar q_*=1$, so $q=1$ and $X$ is simply connected.


Now let us compute the second homology of $X$ over $\RR$.

\begin{proposition} \label{prop:cohom-X}
$$H^2(X,\RR)= \la \o_{\S_2} , dt\wedge ds \ra$$
\end{proposition}
\begin{proof}
First of all one can prove that  $H^2(X,\RR) \cong H^2(M_\g(\S_2) \times S^1,\RR)^{\la \p, \xi \ra }$ by averaging closed forms. Kunneth formula ensures that  
$$
H^2(M_\g(\S_2) \times S^1,\RR)= H^1(M_\g(\S_2),\RR)\wedge \langle ds \rangle \oplus H^2(M_\g(\S_2),\RR).
$$
The first summand is of course equal to $\la dt \wedge ds \ra$; to compute the second we take into account \cite[Lemma 12]{BFM}:
\begin{align*}
H^{2}(M_\g(\S_2)) =& \ker( \rId - \g^* \colon H^2(\S_2, \RR) \to H^2 (\S_2,\RR)) \\
 &\oplus  \coker( \rId - \g^* \colon H^1(\S_2, \RR) \to H^1(\S_2,\RR)) \wedge \langle dt \rangle.
\end{align*}
On the one hand, $\g^*=\rId \colon H^2(\S_2) \to H^2(\S_2)$ because $ \g_*(\o_{\S_2})=\o_{\S_2}$, as was previously argued. On the other,  $\g^*=-\rId \colon H^1(\S_2, \RR) \to H^1(\S_2,\RR)$; this can be deduced from the fact that $\g_*=-\rId$. Thus, 
$$
H^{2}(M_\g(\S_2))= \la \o_{\S^2} \ra.
$$
The proof concludes by observing that both $\o_{\S^2}$ and $dt \wedge ds$ 
are invariant under the action of $\la \p, \xi \ra$.
\end{proof}

\begin{proposition}
Let $\pi \colon \widetilde X \to X$ the symplectic resolution of $X$.
Denote $\S^0=\{p_1,\dots,p_5\}$; then $E_j=\pi^{-1}(p_j)$ is diffeomorphic to $\CP^1$. In addition, 
\begin{enumerate}
\item $\pi_1(\widetilde{X})=\{1\}$.
\item $H^2(\widetilde{X},\RR)= \la \pi^*(\o_{\S_2}), \pi^*(dt\wedge ds), \o_1,\o_2,\o_3,\o_4,\o_5 \ra$,
where $\o_j$ is the Thom class of $E_j$.
\end{enumerate}
\end{proposition}
\begin{proof}
First observe that $\Delta=\emptyset$, where $\Delta$ is defined as in
Proposition \ref{prop:cohom}. In addition, if $p \in \S^0$ is an isolated singularity
then $\G_p= \ZZ_2$; the K\" ahler local model around $p$ is necessarily of the form $\CC^2/ \ZZ_2$, with $\ZZ_2= \la \rId, -\rId \ra$. The algebraic resolution of this space is $\widetilde \CC^2/ \ZZ_2$, where $\widetilde \CC^2$ stands for the blow-up of $0$ in $\CC^2$; that is:
$$
\widetilde \CC^2/ \ZZ_2= \{ (v,l)\in \CC^2 \times \CP^1 \mbox{ s.t. } v\in l \}/ (v,l)\sim (-v,l).
$$

We compute $\pi_1(\widetilde{X})$ using Seifert-Van Kampen theorem. Let $B_j^\e$ be an $\e$-ball centered at $p_j$ with $\e$ small enough to ensure that $B_j^\e$ are pairwise disjoint. Let $N_{j}$ be a neighbourhood of a path between $p_j$ and $p_{j+1}$ that does not intersect $B_j^\e$ for $k\neq j, j+1$. 
Define:
$$
 U= \left( \cup_{j=1}^5{B_j}^\e \right) \cup \left( \cup_{j=1}^4{N_j}\right), \qquad V= X-\cup_{j=1}^5 \overline{B_j}^{\frac{\e}{2}}.
$$
The space $U\cap V$ is pathwise connected and has the homotopy type of $\bigvee_{j=1}^5 S^3_j/\ZZ_2$, where we denoted a copy of $S^3$ as $S^3_j$. Its fundamental group is the free product of 5 copies of $\ZZ_2$. Being $U$ contractible, it holds that $1=\pi_1(X)= \pi_1(V)/ i_*(\pi_1(U\cap V))$, with $i \colon U\cap V \to V$.

In addition define $\widetilde{U}= \pi^{-1}(U)$, $\widetilde{V}=\pi^{-1}(V)$. The space $\widetilde U $ has the homotopy type of $\bigvee_{j=1}^5 \CP^1_j$; which is simply connected. Thus, $\pi_1(\widetilde X)= \pi_1(\widetilde V)/ j_*(\pi_1(\widetilde U\cap \widetilde V))$, with $j \colon \widetilde U\cap \widetilde V \to \widetilde V$. Taking into account that $\pi \colon (\widetilde V, \widetilde U \cap \widetilde V) \to (V, U\cap V)$ is an homeomorphism of pairs; this ensures that $\pi_1(\widetilde X)=\pi_1(X)=\{ 1\}$.

We finally compute $H^2(\widetilde X,\RR)$. By Propositions \ref{prop:cohom} and \ref{prop:cohom-X} there is a short exact sequence:
$$
0 \rightarrow \la \o_{\S_2}, dt \wedge ds \ra \xrightarrow{\pi^*} H^2(\widetilde{X},\RR) \xrightarrow{i^*}  \sum_{j=1}^5 H^2(E_j,\RR) \rightarrow 0.
$$

The restriction of $\o_j$ to $E_j$ is a volume form of $E_j$ because the bundle $\widetilde \CC^2 \to \CP^1$ is non-trivial. This yields a splitting:
$
 i^*(\o_j) \longmapsto \o_j.
$
This finishes the proof.

\end{proof}

\end{document}